\documentclass[10pt,leqno]{article} 

   \usepackage[centertags]{amsmath}
   \usepackage{amsfonts}
   \usepackage{amsmath}
   \usepackage{amssymb}
   \usepackage{amsthm}
   \usepackage{newlfont}
   \usepackage[latin1]{inputenc}

 \usepackage{multirow, bigdelim}

\usepackage{color}

\allowdisplaybreaks[3]

\theoremstyle{plain}
\newtheorem{theorem}{Theorem}[section]
\newtheorem{lemma}[theorem]{Lemma}

\newtheorem{corollary}[theorem]{Corollary}

\theoremstyle{definition}

\theoremstyle{remark}
\newtheorem{remark}[theorem]{Remark}
\newtheorem*{remark*}{Remark}

\numberwithin{equation}{section}

\newcommand{\dosfilas}[2]{
  \ldelim[{2}{2mm}& #1 &\rdelim]{2}{2mm} \\
  & #2 & &  & &
}

\newcommand*\pFqskip{8mu}
\catcode`,\active
\newcommand*\pFq{\begingroup
        \catcode`\,\active
        \def ,{\mskip\pFqskip\relax}%
        \dopFq
}
\catcode`\,12
\def\dopFq#1#2#3#4#5{%
        {}_{#1}F_{#2}\biggl(\genfrac..{0pt}{}{#3}{#4};#5\biggr)%
        \endgroup
}

\newcommand\D{{\mathcal D}}
\newcommand\A{{\mathcal A}}

\newcommand\M{{\mathcal M}}

\newcommand\CC{{\mathbb C}}

\newcommand\ZZ{{\mathbb Z}}
\newcommand\NN{{\mathbb N}}
\newcommand\PP{{\mathbb P}}

\newcommand\Sh{\mbox{\Large $\mathfrak {s}$}}

   \title{Bispectral dual Hahn polynomials with an arbitrary number of continuous parameters.\footnote{Partially supported by PGC2018-096504-B-C31
(FEDER(EU)/Ministerio de Ciencia e Innovaci\'on-Agencia Estatal de Investigaci\'on),
FQM-262 and Feder-US-1254600 (FEDER(EU)/Jun\-ta de Anda\-lu\-c\'ia).}}
   \author{Antonio J. Dur\'{a}n\\
     \footnotesize
        \  Departamento de An\'{a}lisis Matem\'{a}tico.
       Universidad de Sevilla \\
       \footnotesize Apdo (P. O. BOX) 1160. 41080 Sevilla. Spain.
   duran@us.es \\
          \ \ }
   \date{}
   
\begin{document}
   \maketitle

\bigskip

\begin{abstract}
We construct new examples of bispectral dual Hahn polynomials, i.e., orthogonal polynomials with respect to certain superposition of Christoffel and Geronimus transforms of the dual Hahn measure and which are also eigenfunctions of a higher order difference operator.  The new examples have the novelty that they depend on an arbitrary number of continuous parameters. These are the first examples with this property constructed from the classical discrete families.
\end{abstract}

\section{Introduction and results}
The explicit solution of certain mathematical models of physical interest often depends on the use of
special functions. In many cases, these special functions turn out to be certain families of orthogonal
polynomials which, in addition, are also eigenfunctions of second order operators of some specific
kind. These families are the classical, classical discrete and $q$-classical families of orthogonal polynomials. Besides the orthogonality, they are also common eigenfunctions of a second order differential, difference or $q$-difference operator, respectively.
In the terminology introduced by Duistermaat and Gr\"unbaum \cite{DG} (see also \cite{GrH1}, \cite{GrH3})),  they are examples of the so-called bispectral polynomials, because with these families $(q_n(x))_n$ of polynomials are associated two  operators with respect to which they are eigenfunctions: one acting in the discrete variable $n$ (the three term recurrence relation associated to the orthogonality with respect to a measure in the real line) and the other in the continuous variable $x$.

As an extension of the classical families, more than eighty years ago H.L. Krall raised the issue of orthogonal polynomials which are also common eigenfunctions of a higher order differential operator. He obtained a complete classification for the case of a differential operator of order four (\cite{Kr2}). After his pioneer work, orthogonal polynomials which are also common eigenfunctions of higher order differential operators are usually called Krall polynomials (they are also examples of bispectral polynomials).  Since the eighties a lot of effort has been devoted to find Krall polynomials (\cite{koekoe,koe,koekoe2,L1,L2,GrH1,GrHH,GrY,Plamen1,Plamen2,Zh}, the list is by no mean exhaustive). $q$-Krall polynomials were introduced by Grünbaum and Haine in 1996 \cite{GrH2} (see also \cite{VZ,HI,Pl3,dura}).

The problem of finding Krall discrete polynomials was open for some decades. Richard Askey explicitly posed in 1991 (see page 418 of \cite{BGR}) the problem of finding orthogonal polynomials which are also common eigenfunctions of a higher order difference operator (Krall discrete polynomials) of the form
\begin{equation}\label{hodor}
\sum_{l=s}^rh_l\Sh _l, \quad s\le r, s,r\in \ZZ,
\end{equation}
where $h_l$ are polynomials and $\Sh_l$ stands for the shift operator $\Sh_l(p)=p(x+l)$.

The first examples
of discrete Krall polynomials needed more than twenty years to be constructed: a huge amount of families of Krall discrete orthogonal polynomials were introduced by the author by mean of certain Christoffel transforms of the classical discrete measures of Charlier, Meixner, Krawtchouk and Hahn and dual Hahn (see \cite{du0,du1,DdI,DdI2}). A Christoffel transform is a transformation which consists in multiplying a measure $\mu$ by a polynomial $r$. Families of Krall dual Hahn orthogonal polynomials were introduced in \cite{dudh} also by mean of certain Christoffel transforms of the dual Hahn measure.
In the dual Hahn case, for a real number $v$, we consider the linear space $\PP^{\lambda}$ of polynomials in $\lambda(x)=x(x+v+1)$:
$$
\PP^{\lambda}=\{p(\lambda): p\in \PP\}.
$$
The Krall dual Hahn polynomials as functions of $\lambda (x)$ are eigenfunctions of a higher order difference operator of the form
\begin{equation}\label{hodo}
\sum_{j=s}^rh_j\Sh_{j}^\lambda,
\end{equation}
where $h_j$, $j=s,\ldots,r, s\leq r$, are rational functions and the shift operator $\Sh_j^\lambda$ acts in $\PP^{\lambda }$ and it is defined by
$\Sh_j^\lambda (p(\lambda (x)))=p(\lambda (x+j))$.
We denote by $\A^\lambda$ the algebra formed by all the operators $T$ of the form (\ref{hodo}) which maps $\PP^\lambda$ into itself:
\begin{equation}\label{defal}
\A^\lambda=\{T: \mbox{$T$ is of the form (\ref{hodo}) and $T(\PP^\lambda)\subset \PP^\lambda$}\}.
\end{equation}

The  reciprocal of the Christoffel transform is the Geronimus transform. Given a polynomial $r$ and a measure $\mu$, a Geronimus transform $\nu$ of the measure $\mu$ with respect to the polynomial $r$ satisfies $r\nu=\mu$.
Note that a Geronimus transform of the measure $\mu$ with respect to $r$ is not uniquely defined. Indeed, write $a_i$, $i=1,\ldots , u$, for the different real roots of the polynomial $r$, each one with multiplicity $b_i$, respectively.
It is easy to see that if $\nu$ is a Geronimus transform of $\mu$ then the measures $\nu +\sum_{i=1}^u\sum_{j=0}^{b_i-1}\gamma_{i,j}\delta _{a_i}^{(j)}$ are also  Geronimus transforms of $\mu$ with respect to $r$, where $\gamma_{i,j}$ are real numbers which sometimes are called free parameters of the Geronimus transform ($\delta_{a}^{(j)}$ is the signed measure determined by the sequence of moments $(m_n^j)_n$: $m_n^j=0$, $0\le n\le j-1$ and $m_n^j=n(n-1)\cdots (n-j+1)a^{n-j}$, $n\ge j$). In the literature, Geronimus transform is sometimes called Darboux transform with parameters
(see, for instance, \cite{YZ,Z3}).

Krall measures turn out to be Geronimus transforms of the Laguerre and Jacobi weights when the Laguerre parameter $\alpha$ or one or both of the Jacobi parameters $\alpha,\beta$ are nonnegative integers. As a consequence there are families of Krall measures depending on an arbitrary number of continuous parameters. For example, for positive integers $k$ and $m$ and real parameters $M_i$, $i=0\cdots, k$, the measures
$$
\sum_{i=0}^k M_i\delta_0^i+x^me^{-x}dx,\quad x>0,
$$
are Krall measures (\cite{GrHH}).
In the same way, there are examples of $q$-Krall measures depending on an arbitrary number of continuous parameters (because most of the $q$-Krall measures are also Geronimus transform of  $q$-classical measures). However, that did not seem to be the case of the Krall discrete measures: they are Christoffel transform of discrete measures defined from certain finite sets of positive integers, hence besides the continuous parameters of the associated classical discrete measures, all the Krall discrete measures known up to now only depend on an arbitrary number of discrete parameters. Here it is a typical example: for real numbers $a,b$, $a,b>-1$, a positive integer $N$ and a finite set $F=\{f_i:i=1,\cdots , k\}$ of positive integers, the measures
$$
\prod _{f\in F}(x-\lambda ^{a,b}(f))\rho_{a,b,N},
$$
where $\rho_{a,b,N}$ is the dual Hahn measure (see (\ref{masdh})), are Krall discrete measures. Besides the parameters $a,b,N$ of the dual Hahn measures, we have an arbitrary number of discrete parameters: the elements $f_i$, $i=1,\cdots ,k$, of the finite set $F$ which have to be positive integers.

The purpose of this paper is to introduce some new examples of Krall dual Hahn measures. They depend on an arbitrary number of continuos parameters. Each one of this Krall dual Hahn measure is constructed by applying a certain Christoffel transform to a suitable Geronimus transform of a dual Hahn measure with nonnegative parameters $a$ and $b$.
We guess that there are no Krall Charlier, Meixner, Krawtchouk and Hahn measures depending on continuous parameters other that the continuous parameters of the associated classical discrete measure.

The content of the paper is as follows.

In Section \ref{sec2}, we introduce what we call the \lq\lq basic example\rq\rq. Let $a,b,N$ be positive integers with $1\le b\le a\le N$, and set
$$
\lambda^{a,b}(x)=x(x+a+b+1).
$$
Write $\M=\{M_0,\cdots, M_{b-1}\}$ for a finite set consisting of $b$ real parameters, $M_i\not =0,1$. We then define the discrete measure $\nu_{a,b,N}^\M$ supported in the finite quadratic net
$$
\{\lambda^{a,b}(i):i=-b,\cdots ,N\}
$$
by
\begin{align}\label{lanu}
\nu_{a,b,N}^{\M}=&\sum_{x=-b}^{-1}\frac{(2x+a+b+1)(N+1-x)_{x+b}}{(N+b+1)_{x+a+1}}M_{x+b}\delta_{\lambda^{a,b}(x)}\\\nonumber
&\qquad +\frac{(N+1)_b^2}{(b+1)_{a-b}}\sum_{x=0}^N \frac{\rho_{b,a,N}(x)}{\prod_{i=0}^{b-1}(x+a+i+1)(x+b-i)}\delta _{\lambda^{a,b}(x)},
\end{align}
where $\rho_{b,a,N}$ is the dual Hahn measure (see (\ref{masdh}) below).

It is easy to check that the measure $\nu_{a,b,N}^{\M}$ is positive if and only if the parameters in $\M$ are positive. Notice that
$$
\prod_{i=0}^{b-1}(x+a+i+1)(x+b-i)=\prod_{i=0}^{b-1}(\lambda^{a,b}(x)-\lambda^{a,b}(i-b)),
$$
and hence the measure $\nu_{a,b,N}^{\M}$ is a Geronimus transform of the dual Hahn measure $\rho_{b,a,N}$ associated to the polynomial $\prod_{i=0}^{b-1}(x-\lambda^{a,b}(i-b))$:
$$
\prod_{i=0}^{b-1}(x-\lambda^{a,b}(i-b))\nu_{a,b,N}^{\M}=\frac{(N+1)_b^2}{(b+1)_{a-b}}\rho_{b,a,N}.
$$
We find necessary and sufficient conditions for the existence of orthogonal polynomials with respect to the measure $\nu_{a,b,N}^{\M}$ and
construct them explicitly when there exist. We also prove that they are eigenfunctions of a higher order difference operator of the form (\ref{hodo}). For the convenience of the reader we state here the result in full. To do that, we need to introduce some auxiliary functions. As usual, $\lceil x\rceil$ denotes the ceiling function: $\lceil x\rceil=\min \{n\in \ZZ:n\ge x\}$, and $(x)_m$, $m\in \NN$, denotes the Pochhammer symbol $(x)_m=x(x+1)\cdots (x+m-1)$; we also set $(x,y)_m=(x)_m(y)_m$.
For $u\in \NN$, we define
$$
\varphi_{u}^{a,b,N}(s,x)=(-a+1,-N)_{\max(u,a+b-u-1)}\pFq{3}{2}{-u,u-s-a-b+1,-x}{-a-s+1,-N}{1}.
$$
Since $u\in \NN$, except for  normalization, $\varphi_{u}^{a,b,N}(s,x)$ is the Hahn polynomial $h_u^{-a-s,-b,N}(x)$ (\ref{hpol}). Hence as a function of $x$ $\varphi_{u}^{a,b,N}(s,x)$ is a polynomial of degree at most $u$, and as a function of $s$ it is rational and analytic at $s=0$ when $u\le a-1$.

We next define the sequence $(W_g^{a,b,N;\M})_g$ of polynomials, $W_g^{a,b,N;\M}$ of degree $g$, as follows
\begin{equation}\label{losw}
\begin{cases} \frac{\partial}{\partial s}\varphi_{g}^{a,b,-2-N}(0,x)-\frac{\partial}{\partial s}\varphi_{a+b-g-1}^{a,b,-2-N}(0,x) ,& \lceil\frac{a+b}{2}\rceil \le g\le a-1,\vspace{.2cm}\\
 (-1)^{b+g}(g-b)!&\\
\hspace{.2cm}\times \left[(b+a-g-1)!(-x)_ah_{g-a}^{a,-b,-2-N-a}(x-a)\right.&\\
\hspace{.4cm} \left.+\frac{(g-a)!(N+a+b+1-g)_{2g-a-b+1}}{M_{g-a}-1}h_{a+b-g-1}^{-a,-b,-2-N}(x)\right],& a\le g\le a+b-1,\vspace{.2cm}\\
 h_g^{-a,-b,-2-N}(x), & \mbox{otherwise},
 \end{cases}
\end{equation}
where $h_n^{a,b,N}$ denotes the $n$-th Hahn polynomial (see (\ref{hpol}) below). Notice that only the polynomial $W_{i+a}^{a,b,N;\M}$ depends on the parameter $M_i$, $i=0,\cdots , b-1$.
In Section \ref{sec1}, we explain the way in which we have found both, the measure $\nu_{a,b,N}^{\M}$ and the auxiliary polynomials $(W_g^{a,b,N;\M})_g$.

We finally define the sequence $(\Phi_n^{a,b,N;\M})_n$ by
\begin{equation}\label{defphi}
\Phi_n^{a,b,N;\M}=
\left|
  \begin{array}{@{}c@{}lccc@{}c@{}}
  &  &&\hspace{-.9cm}{}_{1\le j\le a} \\
    \dosfilas{W_{g}^{a,b,N;\M}(-n+j-1) }{g\in \{b,b+1,\cdots, a+b-1\}}
  \end{array}\hspace{-.3cm}\right| .
\end{equation}
Throughout  this paper, we use the following notation:
given a finite set of numbers $F=\{f_1,\ldots , f_{n_F}\}$ (we denote by $n_X$ the number of elements of the finite set $X$),  the expression
\begin{equation}\label{defdosf}
  \begin{array}{@{}c@{}lccc@{}c@{}}
  &  &&\hspace{-.9cm}{}_{1\le j\le n_F} \\
    \dosfilas{ z_{f,j}  }{f\in F}
  \end{array}
\end{equation}
inside  of a matrix or a determinant will mean the submatrix defined by
$$
\left(
\begin{array}{cccc}
z_{f_1,1} & z_{f_1,2} &\cdots  & z_{f_1,n_F}\\
\vdots &\vdots &\ddots &\vdots \\
z_{f_{n_F},1} & z_{f_{n_F},2} &\cdots  & z_{f_{n_F},n_F}
\end{array}
\right) .
$$
The determinant (\ref{defphi})  should be understood in this form.

\begin{theorem}\label{th1} Let $a,b,N$ be nonnegative integers with $1\le b\le a\le N$, and write $\M=\{M_0,\cdots, M_{b-1}\}$ for a finite set consisting of $b$ real parameters, $M_i\not =0,1$. Then the measure $\nu_{a,b,N}^{\M}$ has a sequence of orthogonal polynomials if and only if
\begin{equation}\label{cnys}
\Phi_n^{a,b,N;\M}\not =0,\quad n=0,\cdots, N+b+1.
\end{equation}
In that case the sequence of polynomials $(q_n^{a,b,N;\M})_{n=0}^{N+b}$ defined by
\begin{equation}\label{qusmei2}
q_n^{a,b,N;\M}(x)=
\left|
  \begin{array}{@{}c@{}lccc@{}c@{}}
  & \frac{(-1)^{j-1}}{(b+N-n+j)_{a+1-j}}R_{n-j+1}^{a,b,N}(x) &&\hspace{-.6cm}{}_{1\le j\le a+1} \\
    \dosfilas{W_{g}^{a,b,N;\M}(-n+j-2) }{g\in \{b,b+1,\cdots, a+b-1\}}
  \end{array}\hspace{-.3cm}\right| ,
\end{equation}
is orthogonal  with respect to the measure $\nu_{a,b,N}^{\M}$, with norm
\begin{equation}\label{normq}
\langle q_n^{a,b,N;\M},q_n^{a,b,N;\M}\rangle_{\nu_{a,b,N}^{\M}}=\frac{(N+b)!^2\Phi_n^{a,b,N;\M}\Phi_{n+1}^{a,b,N;\M}}{(N+a-n)!(N+b-n)!(N+b-n+1)_a^2}.
\end{equation}
Moreover, the polynomials $q_n^{a,b,N;\M}(\lambda^{a,b}(x))$, $n\ge 0$, are also eigenfunctions of a higher order difference operator of the form (\ref{hodo}) with $-s=r=ab+1$.
\end{theorem}

In particular, if the measure $\nu_{a,b,N}^{\M}$ is positive (that is, all the parameters $M_i$, $i=0,\cdots b-1$, are positive) then the assumption (\ref{defphi}) holds and we can construct orthogonal polynomials with respect to $\nu_{a,b,N}^{\M}$ by using (\ref{qusmei2}).

Given a finite set of complex numbers $U$ such that $u+v\not =-a-b-1$, $u,v\in U$, we consider in Section \ref{sec3} the Christoffel transform of the measure $\nu_{a,b,N}^{\M}$ defined by
\begin{equation}\label{laxm}
\nu_{a,b,N}^{\M,U}=\prod_{u\in U}(x-\lambda^{a,b}(u))\nu_{a,b,N}^{\M}.
\end{equation}
We construct orthogonal polynomials $(q_n^{a,b,N;\M,U})_n$ with respect to $\nu_{a,b,N}^{\M,U}$ by mean of the formula (compared with (\ref{qusmei2}))
\begin{align}\label{qusmei3}
&q_n^{a,b,N;\M,U}(x)\\\nonumber
&\hspace{.0cm}=
\frac{\left|
  \begin{array}{@{}c@{}lccc@{}c@{}}
  & (-1)^{j-1}R_{n+n_U-j+1}^{a,b,N}(x) &&\hspace{-2.6cm}{}_{1\le j\le a+n_U+1} \\
    \dosfilas{(b+N-n-n_U+j)_{a+n_U+1-j}W_{g}^{a,b,N;\M}(-n-n_U+j-2) }{g\in \{b,b+1,\cdots, a+b-1\}}\\
    \dosfilas{(-1)^{j-1}R_{n+n_U-j+1}^{a,b,N}(\lambda^{a,b}(u))}{u\in U}
  \end{array}\hspace{-.7cm}\right|}{\prod_{u\in U}(x-\lambda^{a,b}(u))}
\end{align}
(the renormalization is necessary to avoid division by zero).

The more interesting cases of the Christoffel transforms (\ref{laxm}) is when
$$
\lambda^{a,b}(u)\in \{\lambda^{a,b}(i):\mbox{$-b\le i\le -2$ or $0\le i\le N$}\},\quad u\in U,
$$
because they provide new examples of Krall discrete measures. We prove it in Section \ref{sec4} and to do that we find other determinantal formula for the orthogonal polynomials $(q_n^{a,b,N;\M,U})_n$ different to (\ref{qusmei3}).

This property of having different and nontrivial determinantal representations is typical of the Christoffel transform
of classical discrete measures (see \cite{ductdm}). By nontrivial, we mean that one of such representations
can not be transformed in other different representation just by combining rows
and columns in the corresponding determinants; in particular, the determinants
corresponding to two different representations can have rather different sizes.
In Section \ref{sec5} we find other different determinantal representations for the families of the orthogonal polynomials
with respect to all the Christoffel transforms  (\ref{laxm}).

The new examples of Krall dual Hahn orthogonal polynomials are also interesting by the following reason. It has been shown in \cite{duha} that exceptional Hahn polynomials can be constructed by applying duality (in the sense of \cite{Leo}) to Krall dual Hahn orthogonal polynomials. Passing then to the limit, exceptional Jacobi polynomials can be constructed.
Exceptional and exceptional discrete orthogonal polynomials $p_n$, $n\in X\varsubsetneq \NN$, are complete orthogonal polynomial systems with respect to a positive measure which in addition are eigenfunctions of a second order differential or difference operator, respectively. They extend the  classical families of Hermite, Laguerre and Jacobi or the classical discrete families of Charlier, Meixner and Hahn. The exceptional families have gaps in their degrees, in the
sense that not all degrees are present in the sequence of polynomials, being that the most apparent difference with their classical counterparts.
The last thirteen years have seen a great deal of activity in the area  of exceptional orthogonal polynomials (see, for instance,
\cite{BK,duch,dume,duha,GFGM,GUKM1,GUKM2} (where the adjective \textrm{exceptional} for this topic was introduced),  \cite{GQ,OS3,OS4,STZ}, and the references therein).

In all the examples appeared before 2021 apart from the parameters associated to
the classical and classical discrete weights, only discrete parameters appear in the construction of each exceptional family.
Although very recently, M.A. Garc\'\i a Ferrero, D. G\'omez-Ullate and R. Milson have introduced in \cite{xle} exceptional Legendre polynomials depending on an arbitrary number of continuous parameters. In a subsequent paper, we construct new examples of exceptional Hahn and  Jacobi polynomials using the new examples of Krall dual Hahn polynomials introduced in this paper. These exceptional Hahn and Jacobi polynomials will depend on an arbitrary number of continuous parameters and include as particular cases the exceptional Legendre polynomials in \cite{xle}.

We finish pointing out that the inclusion of the continuous parameters has needed of some new ideas if we compare with some previous papers (especially \cite{DdI} and \cite{dudh}). Anyway, we have omitted those proofs which are similar to some results in \cite{DdI} or \cite{dudh}.

\section{Preliminaries}
Let $\mu$ be a measure (positive or not) with finite moments $\int x^nd\mu$, $n=0,\cdots, 2K$ ($K$ a positive integer or infinity). We say that a sequence of polynomials $(p_n)_{n=0}^K$, $p_n$ of degree $n$, is orthogonal with respect to $\mu$, if
$$
\int p_np_md\mu \begin{cases} =0,&n\not =m,\\ \not =0, &n=m.\end{cases}
$$
For a discrete measure $\rho=\sum_{x=m}^n a_x\delta _{\lambda (x)}$ and $u\in \NN$, we denote by $\tau_u\rho $ the translated measure
\begin{equation}\label{mtr}
\tau _u \rho=\sum_{x=m-u}^{n-u} a_{x+u}\delta_\lambda(x).
\end{equation}
Similarly, given a polynomial $r$, the measure $r\rho$ is defined by
\begin{equation}\label{ctr}
r \rho=\sum_{x=m}^{n} r(\lambda(x)) a_{x}\delta_\lambda(x).
\end{equation}
\bigskip

Consider the set $\Upsilon$  formed by all finite sets of positive
integers:
\begin{equation*}
\Upsilon=\{F:\mbox{$F$ is a finite set of positive integers}\} .
\end{equation*}
We consider the involution $I$ in $\Upsilon$ defined by
\begin{align}\label{dinv}
I(F)=\{1,2,\cdots, \max F\}\setminus \{\max F-f,f\in F\}.
\end{align}
For $F=\emptyset$, we define $\max F=\min F=-1$, and so $I(\emptyset)=\emptyset$.

The definition of $I$ implies that $I^2=Id$.

The set $I(F)$ will be denoted by $G$: $G=I(F)$. Notice that
\begin{equation}\label{spmu}
\max F=\max G,\quad n_G=\max F-n_F+1,
\end{equation}
where $n_F$ and $n_G$ are the number of elements of $F$ and $G$,
respectively.

\bigskip

Given a finite set of numbers $F=\{f_1,\cdots, f_{n_F}\}$, $f_i<f_j$ if $i<j$, we denote by $V_F$ the Vandermonde determinant defined by
\begin{equation}\label{defvdm}
V_F=\prod_{1=i<j=n_F}(f_j-f_i).
\end{equation}
\subsection{Dual Hahn and Hahn polynomials}
We include here basic definitions and facts about dual Hahn and Hahn polynomials, which we will need in the following Sections.

For $a $ and $b$ real numbers, we write
\begin{equation}\label{deflamb}
\lambda^{a,b}(x)=x(x+a+b+1).
\end{equation}

We write $(R_{n}^{a,b,N})_n$ for the sequence of dual Hahn polynomials defined by
\begin{equation}\label{dhpol}
R_{n}^{a,b,N}(x)=\sum _{j=0}^n\frac{(-n)_j(-N+j)_{n-j}(a+j+1)_{n-j}}{n!(-1)^j j!}\prod_{i=0}^{j-1}(x-i(a+b+1+i))
\end{equation}
(see \cite{KLS}, pp, 209-13).
We have taken a different normalization that in \cite{dudh} since
we are going to deal here with the case when $a$ is a negative integer.

Notice that $R_{n}^{a,b,N}$ is always a polynomial of degree $n$.
Using that
$$
(-1)^j\prod_{i=0}^{j-1}(\lambda^{a,b}(x)-i(a+b+1+i))=(-x)_j(x+a+b+1)_j,
$$
we get the hypergeometric representation
$$
R_{n}^{a,b,N}(\lambda^{a,b}(x))=\frac{(a+1)_n(-N)_n}{n!}\pFq{3}{2}{-n,-x,x+a+b+1}{a+1,-N}{1}.
$$
When $a$ and $b$ are positive integers,
the following identity holds for $a\le n$
\begin{equation}\label{dhpn}
\frac{R_{n}^{-a,-b,N}(x+a+b)}{\prod_{i=b}^{a+b-1}(x+a+b-\lambda^{-a,-b}(i))}=
\frac{(n-a)!}{n!}R_{n-a}^{a,b,N-a-b}(x).
\end{equation}
When $N$ is a positive integer and $a ,b \not =-1,-2,\cdots -N $, $a+b \not=-1,\cdots, -2N-1$, the dual Hahn polynomials $R_n^{a,b,N}$, $n=0,\cdots , N$, are
orthogonal with respect to the following measure
\begin{align}\label{masdh}
\rho_{a,b,N}&=\sum _{x=0}^N \frac{(2x+a+b+1)(a+1)_x(-N)_xN!}{(-1)^x(x+a+b+1)_{N+1}(b+1)_xx!}\delta_{\lambda^{a,b}(x)},
\\\label{normedh}
\langle R_n^{a,b,N},R_n^{a,b,N}\rangle &=\frac{(-N)_n^2\binom{a+n}{n}}{\binom{b+N-n}{N-n}},\quad  n=0,\cdots , N.
\end{align}
The measure $\rho_{a,b,N}$ is positive or negative only when either $-1<a,b$ or $a,b<-N$, respectively.

If $N$ is not a nonnegative integer and $a,-b-N-1\not =-1,-2,\cdots$, the dual Hahn polynomials $(R_n^{a,b,N})_n$ are always orthogonal with respect to a signed measure.

We write $(h_{n}^{a,b,N})_n$ for the sequence of Hahn polynomials defined by
\begin{align}\label{hpol}
h_{n}^{a,b,N}(x)&=(a+1)_n(-N)_n\pFq{3}{2}{-n,-x,x+a+b+1}{a+1,-N}{1}\\\nonumber &=\sum _{j=0}^n\frac{(-n)_j(a+b+n+1)_j(-N+j)_{n-j}(a+j+1)_{n-j}(-x)_j}{j!}.
\end{align}
We have taken a different normalization that in \cite{dudh} since
we are going to deal here with the case when $a$ is a negative integer (see \cite{KLS}, 204-8).

The hypergeometric representation of the Hahn and dual Hahn polynomials
show the duality, $n,m\ge 0$
\begin{equation}\label{sdm2b}
\frac{(a+1)_n(-N)_n}{n!} h_{m}^{a,b,N}(n)=(a+1)_m(-N)_m R_{n}^{a,b,N}(\lambda^{a,b}(m)).
\end{equation}
Hahn polynomials are eigenfunctions of the  second order difference operator
\begin{equation}\label{defbc}
\Gamma=a(x)\Sh_{1}-(a(x)+b(x))\Sh_{0}+b(x)\Sh_{-1},\quad \Gamma(h_n^{a,b,N}(x))=\lambda^{a,b}(n)h_n^{a,b,N}(x),
\end{equation}
where
\begin{align*}
a(x)&=(x+a+1)(x-N),\\
b(x)&=x(x-b-N-1).
\end{align*}
Hahn polynomials also satisfies the following identity.
\begin{equation}\label{hcp}
(-1)^nh_{n}^{a,b,N}(x)=h_{n}^{b,a,N}(N-x).
\end{equation}

\section{Finding the pieces of the puzzle}\label{sec1}
In this Section, we explain the way in which we have found both, the measure $\nu_{a,b,N}^{\M}$ (\ref{lanu}) and the auxiliary polynomials $(W_g^{a,b,N;\M})_g$ (\ref{losw}).

In \cite{dudh}, we construct families of Krall dual Hahn polynomials by using Christoffel transforms of the dual Hahn measure. Our starting point here is the following particular case of these families. Let $F$ be a finite set of positive integers. For real numbers $a,b$ and a positive integer $N$
write
\begin{equation}\label{hats}
\hat a=a-\max F -1, \quad \hat b=b-\max F-1,\quad \hat N=N+\max F+1,
\end{equation}
and assume that either $a,b\not \in \NN$, or $a,b\in \NN$, $a,b\le N$  and  $\hat a,\hat b \not \in \{-1,-2,\cdots \}$. Consider the Christoffel transform $\rho_{a,b,N}^F$ of the dual Hahn measure $\rho_{a,b,N}$ (\ref{masdh}) defined by
\begin{equation}\label{mqs}
\rho_{a,b,N}^F=\tau_{\max F+1}\left(\prod _{f\in F}(x-\lambda ^{\hat a,\hat b}(f))\rho_{\hat a,\hat b, \hat N}\right)
\end{equation}
(see (\ref{mtr}) and (\ref{ctr})). We have modified the expression in \cite{dudh} taking into account that
$$
\lambda^{a,b}(x)-\lambda^{a,b}(-\max F-1+f)=\lambda^{\hat a,\hat b}(x+\max F+1)-\lambda^{\hat a,\hat b}(f).
$$
Under the assumption that for $0\le n\le N+n_G+1$
\begin{equation}\label{defphiv}
\Phi_n^{a,b,N;F}(x)=
\left|
  \begin{array}{@{}c@{}lccc@{}c@{}}
  &  &&\hspace{-1.4cm}{}_{1\le j\le n_G} \\
    \dosfilas{h_{g}^{-a,-b,-2-N}(-n+j-1) }{g\in G}
  \end{array}\hspace{-.3cm}\right| \not =0,
\end{equation}
where $G=I(F)$, and $I$ is the involution defined in (\ref{dinv}), we construct in \cite{dudh} orthogonal polynomials with respect to $\rho_{a,b,N}^F$ using the determinantal formula
\begin{equation}\label{qusmeix}
q_n^{a,b,N;F}(x)=
\left|
  \begin{array}{@{}c@{}lccc@{}c@{}}
  & \frac{(-1)^{j-1}}{(b+N-n+j)_{n_G+1-j}}R_{n-j+1}^{a,b,N}(x) &&\hspace{-.6cm}{}_{1\le j\le n_G+1} \\
    \dosfilas{h_{g}^{-a,-b,-2-N}(-n+j-2) }{g\in G}
  \end{array}\hspace{-.3cm}\right| .
\end{equation}

The assumption (\ref{defphiv}) is obviously equivalent to say that the polynomial $q_n^{a,b,N;F}$ has degree $n$.

The most interesting case is when $\rho_{a,b,N}^F$ is a positive measure, in which case the hypothesis (\ref{defphiv}) holds.

The case when $\hat a,\hat b\in \{-1,-2,\cdots \}$ was not considered in \cite{dudh} because the dual Hahn measure
$\rho_{\hat a,\hat b, \hat N}$ is not well defined. Indeed, since the mass at $\lambda^{\hat a,\hat b}(x)$ of the measure $\rho_{\hat a,\hat b, \hat N}$ is given by (see (\ref{masdh}))
$$
\frac{(2x+\hat a+\hat b+1)(\hat a+1)_x(-\hat N)_x\hat N!}{(-1)^x(x+\hat a+\hat b+1)_{\hat N+1}(\hat b+1)_xx!},
$$
the Pochhammer symbols $(\hat b+1)_x$ and $(x+\hat a+\hat b+1)_{N+1}$ in the denominator of $\rho_{\hat a,\hat b, \hat N}(x)$  vanishes for
$x\ge -\hat b$ and $x=0,\cdots , -\hat a-\hat b-1$, respectively. And, for $N$ big enough, all the points $\lambda^{a,b}(x)$, $x\ge -\hat b$ and $x=0,\cdots , -\hat a-\hat b-1$ are in the support of $\rho_{\hat a,\hat b, \hat N}$.
Surprisingly enough, for certain sets $F$, the measure (\ref{mqs}) makes sense even when $\hat a,\hat b\in \{-1,-2,\cdots \}$.

Indeed, assume that $a,b\in \NN$ with $1\le b\le a$. For real numbers $s,M>0$, with $0<s<\max\{1,\vert M\vert \}$, we define
\begin{equation}\label{losas}
a_s=a-s/M,\quad b_s=b+s,
\end{equation}
so that $a_s,b_s\not \in \ZZ$, $\lim_{s\to 0} a_s=a$, $\lim_{s\to 0} b_s=b$ and $\lim_{s\to 0}\frac{a_s-a}{b_s-b}=-1/M$.
When $F_0=\{a,a+1,\cdots, a+b-1\}$, it is not difficult to prove by a careful computation that when $s\to 0$ the measures $\rho_{a_s,b_s,N}^F$ (\ref{mqs})
converges to the positive measure
\begin{align*}
\mu=&c_{a,b}\sum_{x=-b}^{-1}\frac{(2x+a+b+1)(N+1-x)_{x+b}}{(N+b+1)_{x+a+1}}\delta_{\lambda^{a,b}(x)}\\
&\qquad +\frac{c_{a,b}(N+1)_b^2}{M(b+1)_{a-b}}\sum_{x=0}^N \frac{\rho_{b,a,N}(x)}{\prod_{i=0}^{b-1}(x+a+i+1)(x+b-i)}\delta _{\lambda^{a,b}(x)},
\end{align*}
where
$$
c_{a,b}=\frac{(-1)^{a+b+1}(b-1)!(N+b+1)_{a}^2}{(a-1)!},
$$
and $\rho_{b,a,N}(x)$ is the mass at $\lambda^{a,b}(x)$ of the dual Hahn measure $\rho_{b,a,N}$ (see \ref{masdh})). Note that since $\lambda ^{a,b} (x)=\lambda^{a,b}(-x-a-b-1)$, we can move in the measure $\mu$ the mass at the point $x$ to the point $-x-a-b-1$).

Comparing with (\ref{lanu}), we see that the limit measure $\mu$ is the particular case of the measure $\frac{M}{c_{a,b}}\nu_{a,b,N}^\M$, when all the parameters in $\M$ are equal to $M$. This is the way we have found the measure $\nu_{a,b,N}^\M$ (the first piece of the puzzle).

When $1\le b\le a$, it is not difficult to see that $F_0=\{a,a+1,\cdots, a+b-1\}$ is the minimal set having the property that the measures $\rho_{a_s,b_s,N}^F$ have a limit as $s\to 0$ even though the dual Hahn measure  $\rho_{\hat a,\hat b, \hat N}$ is not well defined.

When all the parameters in $\M$ are equal, orthogonal polynomials with respect to the measure $\nu_{a,b,N}^\M$ can be constructed by taking limits in (\ref{qusmeix}). Indeed, since $G_0=I(F_0)=\{b,b+1,\cdots, a+b-1\}$, we have that the polynomials
\begin{equation}\label{qusmeiv}
q_n^{a_s,b_s,N;M}(x)=
\left|
  \begin{array}{@{}c@{}lccc@{}c@{}}
  & \frac{(-1)^{j-1}}{(b_s+N-n+j)_{m+1-j}}R_{n-j+1}^{a_s,b_s,N}(x) &&\hspace{-.6cm}{}_{1\le j\le a+1} \\
    \dosfilas{h_{g}^{-a+s/M,-b-s,-2-N}(-n+j-2) }{g\in \{b,b+1,\cdots, a+b-1\}}
  \end{array}\right| ,
\end{equation}
are orthogonal with respect to the measure  $\rho_{a_s,b_s,N}^{F_0}$ (assuming (\ref{defphiv})).

However, we have to be very careful when taking limit in (\ref{qusmeiv}). Otherwise, by a direct passing to  the limit we would get the polynomials
\begin{equation}\label{qusmeiv2}
\left|
  \begin{array}{@{}c@{}lccc@{}c@{}}
  & \frac{(-1)^{j-1}}{(b+N-n+j)_{m+1-j}}R_{n-j+1}^{a,b,N}(x) &&\hspace{-.6cm}{}_{1\le j\le a+1} \\
    \dosfilas{h_{g}^{-a,-b,-2-N}(-n+j-2) }{g\in \{b,b+1,\cdots, a+b-1\}}
  \end{array}\right| .
\end{equation}
Under the hypotheses $a,b\not \in \NN$ or $a,b\in \NN$ and $a,b\ge \max F+1$, each Hahn polynomial $h_{g}^{-a,-b,-2-N}$, $g\in G$, in (\ref{qusmeix}) has degree $g$, and (\ref{qusmeix}) provides orthogonal polynomials with respect to the measure $\rho_{a,b,N}^{F}$ (forget for a while the assumption (\ref{defphiv})). But in our case, $F_0=\{a,a+1,\cdots, a+b-1\}$, $a, b\le \max F_0$ and then some of the Hahn polynomials $h_{g}^{-a,-b,-2-N}$, $g=b,b+1,\cdots ,a+b-1$, collapse to a polynomial of smaller degree or even to zero, with the consequence that some of the polynomials in (\ref{qusmeiv2}) can also collapse to zero.
In order to avoid that problem, we proceed as follows. Consider the sets $G_0=I(F_0)=\{b,\cdots, a+b-1\}$ and
$$
G_1=\left\{ \left\lceil\frac{a+b}{2}\right\rceil,\cdots, a-1\right\}\subset G_0, \quad G_2=\{a,\cdots, a+b-1\}\subset G_0
$$
(let us remain that $1\le b\le a$). If $g\not \in G_1\cup G_2$, the Hahn polynomial $h_{g}^{-a,-b,-2-N}$ has degree $g$ and the corresponding row in (\ref{qusmeiv2}) will not produce any problem.

If $g\in G_2$, then $h_{g}^{-a,-b,-2-N}=0$ and the corresponding row in (\ref{qusmeiv2}) collapse to zero. We avoid this by using the
polynomial
\begin{equation}\label{lim1}
\lim_{s\to 0}\frac{1}{s}h_{g}^{-a+s/M,-b+s,-2-N}(x).
\end{equation}
It is easy to see that this is a polynomial of degree $g$. More precisely, a careful computation shows that, except for the multiplicative constant
$M/(M-1)$, the limit above coincides with the combination of two Hahn polynomials in the identity (\ref{losw}).
Note that, in (\ref{losw}) we have taken an arbitrary parameter $M_{g-a}$ for each $g$, $a\le g\le a+b-1$. Hence, we conclude
\begin{equation}\label{lim1b}
W_g^{a,b,N;\M}(x)=\frac{M_{g-a}}{M_{g-a}-1}\lim_{s\to 0}\frac{1}{s}h_{g}^{-a+s/M,-b+s,-2-N}(x).
\end{equation}

If $g\in G_1$, only the powers $x^j$, $j=-g+a+b,\cdots, g$, of $h_{g}^{-a,-b,-2-N}$ vanish. Moreover, it is not difficult to see that then
$$
h_{g}^{-a,-b,-2-N}(x)=\frac{h_{g}^{-a,-b,-2-N}(0)}{h_{-g+a+b-1}^{-a,-b,-2-N}(0)}h_{-g+a+b-1}^{-a,-b,-2-N}(x).
$$
Since $g\in G_1$ if and only if $-g+a+b-1\in \{b,\cdots , \lceil\frac{a+b}{2}\rceil-1\}\subset G_0$,
the $(g+1)$-th and $(-g+a+b)$-th rows in (\ref{qusmeiv2}) are proportional and hence the determinant will be zero. We avoid this by changing the polynomial
$h_{g}^{-a+s/M,-b-s,-2-N}$, $g\in G_1$, in the $(g+1)$-th row of the determinant (\ref{qusmeiv}) by the polynomial
$$
h_{g}^{-a+s/M,-b-s,-2-N}-\frac{h_{g}^{-a+s/M,-b-s,-2-N}(0)}{h_{-g+a+b-1}^{-a+s/M,-b-s,-2-N}(0)}h_{-g+a+b-1}^{-a+s/M,-b-s,-2-N}.
$$
Since the polynomial $h_{-g-a-b-1}^{-a+s/M,-b-s,-2-N}$ defines the $(-g-a-b)$-th row of the determinant (\ref{qusmeiv}), the polynomial $q_n^{a_s,b_s,N;M}$ remains the same. Hence, for $g\in G_1$, we consider the polynomial
$$
\lim_{s\to 0}\frac{1}{s}\left(h_{g}^{-a+s/M,-b-s,-2-N}-\frac{h_{g}^{-a+s/M,-b-s,-2-N}(0)}{h_{-g-a-b-1}^{-a+s/M,-b-s,-2-N}(0)}h_{-g-a-b-1}^{-a+s/M,-b-s,-2-N}\right).
$$
It is easy to see that this polynomial is equal to the polynomial
$$
\frac{M}{M-1}W_g^{a,b,N;\M}(x)
$$
(see (\ref{losw})) when the set of parameters is $\M=\{M,\cdots, M\}$ and $\lceil\frac{a+b}{2}\rceil -1\le g\le a-1$. Obviously, the parameter $M$ does not play any role in this case and we have
\begin{equation}\label{lim1c}
\lim_{s\to 0}\frac{1}{s}\left(h_{g}^{-a-s,-b,-2-N}-\frac{h_{g}^{-a-s,-b,-2-N}(0)}{h_{-g+a+b-1}^{-a-s,-b,-2-N}(0)}h_{-g+a+b-1}^{-a-s,-b,-2-N}\right)=W_g^{a,b,N;\M}.
\end{equation}
A careful computation gives the following explicit expression for the polynomial $W_g^{a,b,N;\M}$ in (\ref{losw}) when $\lceil\frac{a+b}{2}\rceil -1\le g\le a-1$:
\begin{align}\label{elhe1}
\frac{(-g,-x)_{a+b-g}}{(-1)^{-g+a+b}}&\sum_{j=0}^{2g-a-b}(j+2+N+a+b-g,j+b-g+1)_{2g-a-b-j}\\\nonumber
&\hspace{2.5cm}\times \frac{(-2g+a+b,-x+a+b-g)_j}{(-g+a+b)\binom{j+a+b-g}{j}}\\\nonumber
&\hspace{-.5cm}+\sum_{j=0}^{a+b-g-1}\frac{(j+2+N,-a+j+1)_{g-j}(-g,g-a-b+1,-x)_j}{j!}\\\nonumber
&\hspace{2.53cm}\times \sum_{i=0}^{j-1}\frac{(2g-a-b+1)}{(-g+i)(g-a-b+1+i)}.
\end{align}
(which shows that it is a polynomial of degree $g$).

This is the way we have found the polynomials $W_g^{a,b,N;\M}$ (the second piece of the puzzle).

When all the parameters in $\M$ are equal, orthogonal polynomials with respect to the measure $\nu_{a,b,N}^\M$ can then be constructed using (\ref{qusmei2}).

In the next Section, we prove that the determinantal formula (\ref{qusmei2}) also works to construct orthogonal polynomials with respect to the measure $\nu_{a,b,N}^\M$ in the general case of arbitrary parameters $\M=\{M_0,\cdots, M_{b-1}\}$.

\section{The basic example}\label{sec2}
In this Section, we prove Theorem \ref{th1}. Our starting point are the two positive integers $a,b$ with $1\le b\le a$.

We need to introduce some auxiliary functions. Firstly, we define the polynomial $P$ as follows
\begin{equation}\label{elpp}
P(x)=\prod_{j=b}^{a+b-1}(x+\lambda^{a,b} (-j-1)).
\end{equation}
It is easy to see that if $i=b,\cdots, a-1$, then $-\lambda^{a,b} (-i-1)$ is a double root of $P$. Define then the polynomial $P_i$ by
\begin{equation}\label{elppi}
P_i(x)=\frac{(2i+1-a-b)P(x)}{(x+\lambda^{a,b} (-i-1))^2}.
\end{equation}
Since $\lambda^{a,b} (-i-1)=\lambda^{a,b} (i-a-b)$, we get that
$P_i=-P_{a+b-1-i}$ when
\begin{equation}\label{lascr}
\mbox{either $i=b,\cdots, \lceil\frac{a+b}{2}\rceil-2$ or $i=\lceil\frac{a+b}{2}\rceil-1$ and $a+b=2\lceil\frac{a+b}{2}\rceil$.}
\end{equation}
Define also the numbers $u^m_i$, $m\ge 0$, by
\begin{equation}\label{losuim}
u^m_i=\begin{cases} \frac{m(\lambda^{a,b} (-i-1))^{m-1} P_i(-\lambda^{a,b} (-i-1))}{P_i'(-\lambda^{a,b} (-i-1))},& \mbox{if $i$ satisfies  (\ref{lascr}),}\\
0,& \mbox{otherwise.}\end{cases}
\end{equation}
For $g=b,\cdots, a+b-1$, we finally define the sequences $(\psi^m_g)_m$ as follows
\begin{equation}\label{lasphi}
\psi^m_g=\begin{cases}\frac{\left((\lambda^{a,b} (-g-1))^m +u_g^m\right)\operatorname{Res}_{-\lambda^{a,b} (-g-1)}(1/P)}{W_g^{a,b,N;\M}(0)}, &g=b,\cdots,\lceil\frac{a+b}{2}\rceil -1, \\\vspace{.1cm}
\frac{(\lambda^{a,b} (-g-1))^m}{P_g(-\lambda^{a,b} (-g-1))h_g^{-a,-b,-2-N}(0)}, &g=\lceil\frac{a+b}{2}\rceil ,\cdots,a-1,\\\vspace{.1cm}
\frac{(\lambda^{a,b} (-g-1))^m \operatorname{Res}_{-\lambda^{a,b} (-g-1)}(1/P)}{W_g^{a,b,N;\M}(0)}, &g=a,\cdots,a+b-1.
\end{cases}
\end{equation}
The proof of Theorem \ref{th1} is based in the following identities, which we will prove later on.

\begin{lemma}\label{lem2} Let $a,b,N$ be nonnegative integers with $1\le b\le a\le N$, and write $\M=\{M_0,\cdots, M_{b-1}\}$ for a finite set consisting of $b$ real parameters, $M_i\not =0,1$.
For $0\le n $,  $0\le m\le n$ and $m-a+1\le s\le n$  we have
\begin{equation}\label{lasff}
\frac{(-1)^{a+s+1}\langle R_{s}^{a,b,N},x^m\rangle _{\nu_{a,b,N}^\M}}{(a-1)!(N+2)_{b-1}}=(b+N-s+1)_{s}\sum_{g=b}^{a+b-1}\psi^m_g W_g^{a,b,N;\M}(-s-1),
\end{equation}
and for $n=0,1,\cdots ,N+a$
\begin{align}\label{lasff2}
&\frac{(-1)^{n+1}\langle R_{n-a}^{a,b,N},x^n\rangle _{\nu_{a,b,N}^\M}}{(a-1)!(N+2)_{b-1}(b+N-n+a+1)_{n-a}}
=\frac{(-1)^{n+1}n!(N+1)!}{(a-1)!(N+a-n)!}\\\nonumber&\hspace{5cm}+\sum_{g=b}^{a+b-1}\psi^n_g W_g^{a,b,N;\M}(-n+a-1).
\end{align}
As a consequence we have that $\Phi_0^{a,b,N;\M}\not =0$ (see (\ref{defphi})).
\end{lemma}

\begin{proof}[Proof of Theorem \ref{th1}]
We proceed in four steps.

\noindent
\textit{Step 1.} Assume first that $\Phi_{n}^{a,b,N;\M}\not =0$  for a certain $n$, $0\le n \le N+b$. Then the polynomial $q_n^{a,b,N;\M}(x)$ (\ref{qusmei2}) has degree $n$, $q_n^{a,b,N;\M}(x)$ and $x^m$, $m=0\cdots, n-1$, are orthogonal with respect to the measure $\nu_{a,b,N}^\M$ and its norm is given by (\ref{normq}).

Indeed, since the leading coefficient of $q_n^{a,b,N;\M}$ is
\begin{equation}\label{lcq}
\frac{1}{(b+N-n+1)_an!}\Phi_{n}^{a,b,N;\M},
\end{equation}
we deduce that the polynomial $q_n^{a,b,N;\M}(x)$  has degree $n$. Then
\begin{align*}
\langle q_n^{a,b,N;\M}(x), & x^m\rangle_{\nu_{a,b,N}^\M}\\&=\left|
  \begin{array}{@{}c@{}lccc@{}c@{}}
  & \frac{(-1)^{j-1}}{(b+N-n+j)_{a+1-j}}\langle R_{n-j+1}^{a,b,N}(x), x^m\rangle_{\nu_{a,b,N}^\M}
    &&\hspace{-.5cm}{}_{1\le j\le a+1} \\
    \dosfilas{W_{g}^{a,b,N;\M}(-n+j-2) }{g\in \{b,b+1,\cdots, a+b-1\}}
  \end{array}\hspace{-.3cm}
 \right|.
\end{align*}
For $m=0,\cdots, n-1$, $s=n-j+1$, the identities (\ref{lasff}) in Lemma \ref{lem2} show that the first row of the determinant above is a linear combination of the following rows. Hence, the determinant vanishes and we deduce that
$$
\langle q_n^{a,b,N;\M}(x),x^m\rangle_{\nu_{a,b,N}^\M}=0.
$$
That is, the polynomials $q_n^{a,b,N;\M}$ and $x^m$, $m=0,\cdots , n-1$, are orthogonal with respect to the measure $\nu_{a,b,N}^\M$.

For $m=n$, combining the rows of the determinant above using  the identities (\ref{lasff}) for $m=n$, $s=n-j+1$ and $j=1,\cdots ,a$, and (\ref{lasff2}) in Lemma \ref{lem2}, we get
\begin{align*}\nonumber
\langle q_n^{a,b,N;\M}&(x),x^n\rangle_{\nu_{a,b,N}^\M}\\&=\frac{n!(N+b)!^2}{(N+a-n)!(N+a+b-n)!}\left|
  \begin{array}{@{}c@{}lccc@{}c@{}}
  &     &&\hspace{-1cm}{}_{1\le j\le a} \\
    \dosfilas{W_{g}^{a,b,N;\M}(-n+j-2) }{g\in \{b,b+1,\cdots, a+b-1\}}
  \end{array}
 \hspace{-.3cm}\right| \\\label{xx2}
  &=\frac{n!(N+b)!^2}{(N+a-n)!(N+a+b-n)!}\Phi_{n+1}^{a,b,N;\M},
\end{align*}
from where the identity (\ref{normq}) can be obtained by taking into account the expression (\ref{lcq}) for the leading coefficient of $q_n^{a,b,N;\M}$.
This proves the first step.

\vspace{.4cm}

\noindent
\textit{Step 2.} If (\ref{cnys}) holds then the polynomials $q_n^{a,b,N;\M}(x)$  are orthogonal with respect to the measure $\nu_{a,b,N}^\M$.

It is straightforward from Step 1.

\vspace{.4cm}

\noindent
\textit{Step 3.} If the measure $\nu_{a,b,N}^\M$ has a sequence $(p_n)_{n=0}^{N+b}$ of orthogonal polynomials, then the assumption (\ref{cnys}) holds.

We prove it using induction on $n$.

Lemma \ref{lem2} shows that $\Phi_{0}^{a,b,N;\M}\not =0$.

Assume now that $\Phi_{n}^{a,b,N;\M}\not =0$. Using Step 1, we deduce that the polynomial $q_n^{a,b,N;\M}$ has degree $n$, and hence
$$
q_n^{a,b,N;\M}(x)=\zeta_np_n(x)+\sum_{j=0}^{n-1}\zeta_jp_j(x),
$$
with $\zeta_n\not =0$. Step 1 also gives that $q_n^{a,b,N;\M}$ and $x^m$, $m=0\cdots, n-1$, are orthogonal with respect to the measure $\nu_{a,b,N}^\M$, and since the polynomials $(p_j)_j$ are also orthogonal, we get
$$
\langle q_n^{a,b,N;\M},q_n^{a,b,N;\M}\rangle_{\nu_{a,b,N}^\M}=\zeta_n^2\langle p_n,p_n\rangle_{\nu_{a,b,N}^\M}\not =0.
$$
Finally, Step 1 also says that the non null norm of $q_n^{a,b,N;\M}$ is given by (\ref{normq}), from where we deduce
that also $\Phi_{n+1}^{a,b,N;\M}\not =0$.

\vspace{.4cm}

\noindent
\textit{Step 4.} The polynomials  $q_n^{a,b,N;\M}(\lambda^{a,b}(x))$, $n\ge 0$ are eigenfunctions of a higher order difference operator of the form (\ref{hodo}) (where $\lambda(x)=x(x+a+b+1)$) with $-s=r=ab+1$. (Notice that now $n$ runs over the nonnegative integers).

This is a direct consequence of  \cite[Theorem 3.1]{dudh} (after a suitable renor\-ma\-lization of the polynomials) and the formulas for the $\D$-operators of the dual Hahn polynomials displayed in \cite[Section 6]{dudh}. Indeed, assume that the sequence of polynomials $(p_n)_n$ are eigenfunctions of an operator $D\in \A$, where $\A$ is an algebra of operators acting in the linear space of polynomials. Assume also that the sequence $(\epsilon_n)_n$ defines a $\D$-operator for the sequence of polynomials $(p_n)_n$ and the algebra $\A$  (for the definition of a $\D$-operator, see \cite[Section 3]{du1}, or also \cite[Section 3]{DdI} or \cite[Section 3.1]{dudh}) and write $\xi_{n,i}=\prod_{j=0}^{i-1}\epsilon_{n-j}$. Then  \cite[Theorem 3.2]{DdI} (see also  \cite[Theorem 3.1]{dudh}) states that for any finite set of polynomials $Y_i$, $i=1,\cdots , m$, the polynomials
\begin{equation}\label{qusmeig}
P_n(x)=\left|
  \begin{array}{@{}c@{}lccc@{}c@{}}
  & (-1)^{j-1}p_{n-j+1}(x)/\xi _{n-j+1,m-j+1}
    &&\hspace{-.5cm}{}_{1\le j\le m+1} \\
    \dosfilas{Y_i(n-j+1)}{i=1,\cdots, m}
  \end{array}\hspace{-.3cm}
 \right|
\end{equation}
are also eigenfunctions of an operator in $\A$ (which can be explicitly constructed). In our case, $(p_n(x))_n$ is the sequence of dual Hahn polynomials $(R_n^{a,b,N}(\lambda^{a,b}(x)))_n$, $\A$ is the algebra $\A^{\lambda}$ (\ref{defal}), and
$$
\epsilon_n=b+N-n+1,\quad \xi_{n,i}=(b+N-n+1)_i.
$$
Since the polynomials $q_n^{a,b,N;\M}(\lambda^{a,b}(x))$, $n\ge 0$, have the form (\ref{qusmeig}), they are eigenfunctions of an operator of the form (\ref{hodo}). The order can be computed as in  \cite[Theorem 3.1]{dudh}.

\bigskip

The proof of Theorem \ref{th1} is now complete.
\end{proof}

We complete this Section with the proof of Lemma \ref{lem2}.

\begin{proof}[Proof of Lemma \ref{lem2}]
We first prove that $\Phi_{0}^{a,b,N;\M}\not =0$.
In order to do that we consider the $a\times a$ matrices defined by
$$
\Psi =\left(\begin{array}{@{}c@{}lccc@{}c@{}}
  &     &&\hspace{-2cm}{}_{b\le g\le a+b-1} \\
    \dosfilas{\psi_{g}^{a-i} }{i\in \{1,2,\cdots, a\}}
\end{array}\hspace{-.6cm}\right),\quad \tilde \Phi_0 =\left(\begin{array}{@{}c@{}lccc@{}c@{}}
  &     &&\hspace{-2cm}{}_{1\le l\le a} \\
    \dosfilas{W_{g}^{a,b,N;\M}(l-1) }{g\in \{b,b+1,\cdots, a+b-1\}}
  \end{array}\hspace{-.6cm}\right).
$$
Comparing with (\ref{defphi}), we deduce that $\det \tilde \Phi_0=\Phi_{0}^{a,b,N;\M}.$

Using (\ref{lasff}), we have on the one hand that the entry $(i,l)$, $1\le l<i\le a$, of the matrix product $\Psi\tilde \Phi_0$ is
$$
\sum_{g=b}^{a+b-1}\psi_g^{a-i}W_{g}^{a,b,N;\M}(l-1)=0
$$
(because $-l\le -1$ and then $R_{-l}^{a,b,N}=0$). On the other hand, using (\ref{lasff2}), we have that the entry $(l,l)$, $l=1,\cdots, a$, of the matrix product $\Psi\tilde \Phi_0$ is
$$
\sum_{g=b}^{a+b-1}\psi_g^{a-l}W_{g}^{a,b,N;\M}(l-1)=\frac{(-1)^{a-l}(a-l)!(N+1)!}{(a-1)!(N+l)!}\not =0.
$$
Hence the matrix product $\Psi\tilde \Phi_0$ is upper triangular with non null entries in its diagonal, and then its determinant is different to zero. This implies that also $0\not =\det \tilde \Phi_0=\Phi_{0}^{a,b,N;\M}$.

We next prove the identities (\ref{lasff}) and (\ref{lasff2}).

Given a finite set $F$ of positive integer, under the assumption $a,b\ge \max F+1$,
we proved in \cite[Theorem 5.1]{dudh}, the orthogonality of the polynomials (\ref{qusmeix}) with respect to the measure (\ref{mqs}) by using the following identities (actually, the measure (\ref{mqs}) is a very particular case of \cite[Theorem 5.1]{dudh}). We first introduce some notation.
Write $p$ for the polynomial
\begin{equation}\label{elp}
p(x)=\prod_{g\in G}(x+\lambda^{a,b} (-g-1)),
\end{equation}
where $G=I(F)$ (see (\ref{dinv})). The assumption $a,b\ge \max F+1$ implies that the roots of $p$ are simple.
Write $(\tilde\psi _g^m)_m$, $g\in G$, for the sequences defined by
\begin{equation}\label{lastpsi}
\tilde \psi _g^m =\frac{(\lambda^{a,b} (-g-1))^m}{p'(-\lambda^{a,b}(-g-1))h_g^{-a,-b,-2-N}(0)}.
\end{equation}
Then for $n=0,1,\cdots $,  $0\le m\le n$ and $m-n_G+1\le s\le n$, we have
\begin{equation}\label{lasid}
\langle R^{a,b,N}_{s},x^m\rangle_{\rho^F_{a,b,N}}=\frac{(b+N-s+1)_{s+1}}{(-1)^{s}c_{a,b,N}^F}\sum_{g\in G}\tilde \psi_g^mh_g^{-a,-b,-2-N}(-s-1),
\end{equation}
where
$$
c_{a,b,N}^F=\frac{(-1)^{n_G+1}(b-\max F)_{N+\max F+2}(N+1)!}{(a-\max F)_{\max F}(N+\max F+1)!^2};
$$
and for $n=0,\cdots, N+n_G$, we have
\begin{align}\label{lasid2}
\frac{(-1)^{n-n_G}c_{a,b,N}^F\langle R^{a,b,N}_{n-n_G},x^n\rangle_{\rho^F_{a,b,N}}}{(b+N-n+n_G+1)_{n-n_G-1}}&=\frac{(-1)^{n+1}(n+a-n_G)!(N+1)!}{(a-1)!(N+n_G-n)!}\\\nonumber
&\hspace{-.5cm} +\sum_{g\in G}\tilde \psi_g^nh_g^{-a,-b,-2-N}(-n+n_G-1).
\end{align}

Identities of the type (\ref{lasid}) and (\ref{lasid2}) appear in all the families of Krall-discrete polynomials (see \cite[p. 69, 77]{DdI}, \cite[p. 380-381]{DdI2} for the Krall Charlier, Krall Meixner and Krall Hahn polynomials, respectively). In each one of these identities appears certain family of polynomials in the right hand side (the Hahn polynomials $h_g^{-a,-b,-2-N}(-x-1)$ in the above identities (\ref{lasid}) and (\ref{lasid2})). The polynomials in each one of these families are eigenfunctions of a second order difference operator. These identities can then by proved from the case $m=0$ by induction on $m$ using the associated second order difference operator (see \cite[Section 4]{DdI}, especially Lemmas 4.1 and 4.2).
In particular the identities (\ref{lasid}) and (\ref{lasid2}) follows from the case $m=0$, by induction on $m$ using that the Hahn polynomials $h_g^{-a,-b,-2-N}(-x-1)$, $g\ge 0$, are eigenfunctions of the second order difference operator (see (\ref{defbc}))
\begin{equation}\label{elsoph}
D=A(x)\Sh_{-1}+B(x)\Sh_0+C(x)\Sh_{1},
\end{equation}
where
\begin{align*}
A(x)&=(x+1)(x-b-N),\quad C(x)=(x-N-1)(x+a), \\\nonumber
B(x)&=-A(x-1)-C(x+1),
\end{align*}
and $D(h_g^{-a,-b,-2-N}(-x-1))=\lambda^{a,b}(-g-1)h_g^{-a,-b,-2-N}(-x-1)$ (note that the eigenvalues $\lambda^{a,b}(-g-1)$ define the polynomial $p$ (\ref{elp})).

The identities (\ref{lasff}) and (\ref{lasff2}) in Lemma \ref{lem2} can be proved in a similar way, but taking into account the following two remarks (let us remain that in our case $F=\{a,a+1,\cdots, a+b-1\}$, $G=I(F)=\{b,b+1,\cdots, a+b-1\}$, and hence the assumption $a,b\ge \max F+1$ fails).

\medskip

\begin{remark}\label{urm} For $g\not \in \{ \lceil\frac{a+b}{2}\rceil,\cdots ,a-1\}$, the polynomial $P$ (\ref{elpp}) plays in the  sequences $(\psi_g^m)_m$ (\ref{lasphi}) the same role
played by the polynomial $p$ (\ref{elp}) in the sequences  $(\tilde\psi _g^m)_m$ (\ref{lastpsi}). However, since $P$ can have double roots  (when $b<a-1$), we have to consider in each double root $z_0$ of $P$ the residue $\operatorname{Res}_{z_0}(1/P)$ instead of $1/p'(z_0)$.
\end{remark}
\medskip

\begin{remark}\label{urm2} The polynomial $W_g^{a,b,N;\M}$ (\ref{losw}) plays in the  sequences $(\psi_g^m)_m$ (\ref{lasphi}) the same role
played by the Hahn polynomial $h_g^{-a,-b,-2-N}$  in the sequences  $(\tilde\psi _g^m)_m$ (\ref{lastpsi}). Since for $g\not \in \{ \lceil\frac{a+b}{2}\rceil,\cdots ,a+b-1\}$,  $W_g^{a,b,N;\M}=h_g^{-a,-b,-2-N}$, the polynomial $W_g^{a,b,N;\M}$ is also an eigenfunction of the second order difference operator $D$ (\ref{elsoph}):
$$
D(W_g^{a,b,N;\M}(-x-1))=\lambda^{a,b}(-g-1)W_g^{a,b,N;\M}(-x-1).
$$
This also happens for $g \in \{a,\cdots ,a-1\}$ (it can be easily proved from the limit expression (\ref{lim1b})
in Section \ref{sec1}).

However, for $g\in \{ \lceil\frac{a+b}{2}\rceil,\cdots ,a-1\}$, a non linear term appears
$$
D(W_g^{a,b,N;\M}(-x-1))=\lambda^{a,b}(-g-1)W_g^{a,b,N;\M}(-x-1)+\varsigma h_{a+b-g-1}^{-a,-b,-2-N}(-x-1),
$$
where $\varsigma=(a+b-2g-1)(b-g)_{2g-a-b+1}(N+a+b-g+1)_{2g-a-b+1}$ (this follows from the limit expression
(\ref{lim1c}),
see again Section \ref{sec1}).
Taking this into account, we have to introduce the number $u^m_g$ in $\psi_g^m$, $g\in \{b,\cdots, \lceil\frac{a+b}{2}\rceil-1\}$, and change the factor
$$
\frac{\operatorname{Res}_{\lambda^{a,b} (-g-1)}(1/P)}{W_g^{a,b,N;\M}(0)}
$$
(which appears in $\psi_g^m$ when $g\not \in \{\lceil\frac{a+b}{2}\rceil ,\cdots, a-1\}$)
to the factor
$$
\frac{1}{P_g(-(\lambda^{a,b} (-g-1))h_g^{-a,-b,-2-N}(0)}
$$
when $g\in \{\lceil\frac{a+b}{2}\rceil ,\cdots, a-1\}$.
\end{remark}

Taking into account these remarks the induction to prove Lemma \ref{lem2} works as in \cite[Lemma 4.1]{DdI}.

\end{proof}

\section{Christoffel transforms of the basic example}\label{sec3}
In this Section, we extend the formula (\ref{qusmei2}) for orthogonal polynomials with respect to Christoffel transforms of the measure
$\nu_{a,b,N}^{\M}$.

Let $U$ be a finite set  of complex numbers with $n_U$ elements arranged in increasing lexicographic order (as usual, we define the lexicographic order in $\CC$ by $u<w$ if either $\Re u<\Re w$ or $\Re u=\Re w$ and $\Im u<\Im w$). We associate to $U$ the Christoffel transform of the measure $\nu_{a,b,N}^{\M}$ defined by
\begin{equation}\label{lctnu}
\nu_{a,b,N}^{\M,U}=\prod_{u\in U}(x-\lambda^{a,b}(u))\nu_{a,b,N}^{\M}.
\end{equation}
In order to avoid the polynomial $\prod_{u\in U}(x-\lambda^{a,b}(u))$ from having double roots, we assume along this Section that
\begin{equation}\label{hipct}
u+v\not =-a-b-1, \quad u,v\in U.
\end{equation}
Note that the support of  $\nu_{a,b,N}^{\M,U}$ is the set
$$
S_U^{b,N}=\{\lambda^{a,b}(i):i=-b,\cdots ,N\}\setminus \{\lambda^{a,b}(u):u\in U\}
$$
We denote by $n_S$ the number of elements of $S_U^{b,N}$, and assume that $n_S>0$.

It is easy to see that when $u$ is in
\begin{equation}\label{npc}
U_p=U\cap [\{-a-b+1,\cdots, -a-1\}\cup \{-b,\cdots ,-1\}],
\end{equation}
the factor $x-\lambda^{a,b}(u)$ kills the mass of $\nu_{a,b,N}^{\M}$ at $\lambda^{a,b}(u)$. Hence if we write
$n_-$ for the number of elements of $U_p$, the measure $\nu_{a,b,N}^{\M,U}$ dependes on the $b-n_-$ continuous parameters
$M_{j}$, where $j\not \in U_p$ and $-b\le j\le -1$.

We define the sequence $(\Phi_n^{a,b,N;\M,U})_n$ by
\begin{equation}\label{defphiu}
\Phi_n^{a,b,N;\M,U}=
\left|
  \begin{array}{@{}c@{}lccc@{}c@{}}
  &  &&\hspace{-2.6cm}{}_{1\le j\le a+n_U} \\
    \dosfilas{(b+N-n-n_U+j+1)_{a+n_U-j}W_{g}^{a,b,N;\M}(-n-n_U+j-1) }{g\in \{b,b+1,\cdots, a+b-1\}}\\
    \dosfilas{(-1)^{j}R_{n+n_U-j}^{a,b,N}(\lambda^{a,b}(u))}{u\in U}
  \end{array}\hspace{-.6cm}\right| .
\end{equation}
Note that $\displaystyle \Phi_n^{a,b,N;\M}=\frac{\Phi_n^{a,b,N;\M,\emptyset }}{\prod_{j=1}^a(b+N-n+j+1)_{a-j}}$. The renormalization is necessary to avoid division by zero when $U\not =\emptyset$.

The main result of this Section is the following Theorem.

\begin{theorem}\label{th2} Let $a,b,N$ be nonnegative integers with $1\le b\le a\le N$,  write $\M=\{M_0,\cdots, M_{b-1}\}$ for a finite set consisting of $b$ real parameters, $M_i\not =0,1$, and $U$ for a finite set of complex numbers satisfying (\ref{hipct}) and $n_S>0$.
Then the measure $\nu_{a,b,N}^{\M,U}$ has a sequence of orthogonal polynomials if and only if
\begin{equation}\label{cnysu}
\Phi_n^{a,b,N;\M,U}(n)\not =0,\quad n=0,\cdots, n_S.
\end{equation}
In that case the sequence of polynomials $(q_n^{a,b,N;\M,U})_{n=0}^{n_S-1}$ defined by
\begin{equation}\label{qusmei2u}
q_n^{a,b,N;\M,U}(x)=
\frac{\left|
  \begin{array}{@{}c@{}lccc@{}c@{}}
  & (-1)^{j-1}R_{n+n_U-j+1}^{a,b,N}(x) &&\hspace{-2.6cm}{}_{1\le j\le a+n_U+1} \\
    \dosfilas{(b+N-n-n_U+j)_{a+n_U+1-j}W_{g}^{a,b,N;\M}(-n-n_U+j-2) }{g\in \{b,b+1,\cdots, a+b-1\}}\\
    \dosfilas{(-1)^{j-1}R_{n+n_U-j+1}^{a,b,N}(\lambda^{a,b}(u))}{u\in U}
  \end{array}\hspace{-.6cm}\right|}{\prod_{u\in U}(x-\lambda^{a,b}(u))} ,
\end{equation}
is orthogonal  with respect to the measure $\nu_{a,b,N}^{\M,U}$, with norm
\begin{align}\label{normqu}
&\langle q_n^{a,b,N;\M,U},q_n^{a,b,N;\M,U}\rangle_{\nu_{a,b,N}^{\M,U}}\\\nonumber
&\quad =\frac{(-1)^{n_U}n!(N+b)!^2(N+a+b-n)^{a}}
{(n+n_U)!(N+a-n)!(N+a+b-n)!}
\Phi_n^{a,b,N;\M,U}\Phi_{n+1}^{a,b,N;\M,U}.
\end{align}
\end{theorem}

\begin{proof}
First of all, we note that $q_n^{a,b,N;\M,U}(x)$ is a polynomial because the determinant in the numerator of the right hand side of (\ref{qusmei2u}) vanishes for
$x=\lambda^{a,b}(u)$, $u\in U$.

The proof is analogous to that of Theorem \ref{th1}. We proceed in three steps.

\noindent
\textit{Step 1.} Assume first that $\Phi_{n}^{a,b,N;\M,U}\not =0$  for a certain $n$, $0\le n \le n_S-1$. Then the polynomial $q_n^{a,b,N;\M,U}(x)$ (\ref{qusmei2u}) has degree $n$, it is orthogonal to $x^m$, $m=0\cdots, n-1$, with respect to the measure $\nu_{a,b,N}^{\M,U}$ and its norm is given by (\ref{normqu}).

Indeed, since the leading coefficient of $q_n^{a,b,N;\M,U}$ is
\begin{equation}\label{lcqu}
\frac{1}{(n+n_U)!}\Phi_{n}^{a,b,N;\M,U},
\end{equation}
we deduce that the polynomial $q_n^{a,b,N;\M,U}(x)$  has degree $n$. Then
\begin{align*}
\langle &q_n^{a,b,N;\M,U}(x), x^m\rangle_{\nu_{a,b,N}^{\M,U}}\\&=\left|
  \begin{array}{@{}c@{}lccc@{}c@{}}
  & (-1)^{j-1}\langle R_{n+n_U-j+1}^{a,b,N}, x^m\rangle_{\nu_{a,b,N}^\M}
    &&\hspace{-2.6cm}{}_{1\le j\le a+n_U+1} \\
    \dosfilas{(b+N-n-n_U+j)_{a+n_U+1-j}W_{g}^{a,b,N;\M}(-n-n_U+j-2) }{g\in \{b,b+1,\cdots, a+b-1\}}\\
    \dosfilas{(-1)^{j-1}R_{n+n_U-j+1}^{a,b,N}(\lambda^{a,b}(u))}{u\in U}
  \end{array}
 \hspace{-.6cm}\right|.
\end{align*}
For $m=0,\cdots, n-1$, $s=n+n_U-j+1$,  the identities (\ref{lasff}) in Lemma \ref{lem2} show that the first row of the determinant above is a linear combination of the following $a$ rows (the rows defined by the polynomials $W_g^{a,b,n;\M}$). So, the determinant vanishes and we deduce that
$$
\langle q_n^{a,b,N;\M,U}(x),x^m\rangle_{\nu_{a,b,N}^{\M,U}}=0.
$$
Hence,the polynomials $q_n^{a,b,N;\M,U}$ and $x^m$, $m=0,\cdots , n-1$, are orthogonal with respect to the measure $\nu_{a,b,N}^{\M,U}$.

The identity (\ref{normqu}) for the norm can be proved similarly.

\noindent
\textit{Step 2.} If (\ref{cnysu}) holds then the polynomials $q_n^{a,b,N;\M,U}(x)$  are orthogonal with respect to the measure $\nu_{a,b,N}^{\M,U}$.

It is straightforward from Step 1.

\noindent
\textit{Step 3.} If the measure $\nu_{a,b,N}^{\M,U}$ has a sequence $(p_n)_{n=0}^{n_S}$ of orthogonal polynomials, then the assumption (\ref{cnysu}) holds.

We prove it using induction on $n$.

For $n=0$, we consider the $n_U\times n_U$ determinant defined by
$$
\Lambda =\left|\begin{array}{@{}c@{}lccc@{}c@{}}
  &     &&\hspace{-2cm}{}_{j=1,\cdots, n_U} \\
    \dosfilas{(-1)^{j}R_{n_U-j}^{a,b,N}(\lambda^{a,b}(u))}{u\in U}
\end{array}\hspace{-.6cm}\right|.
$$
It is not difficult to see that, up to a non null factor, $\Lambda$ is equal to the Vandermonde determinant $V_X$ of the finite set $X=\{\lambda^{a,b}(u):u\in U\}$ (see (\ref{defvdm})). The assumption (\ref{hipct}) implies that this Vandermonde determinant is different to zero, and hence $\Lambda\not =0$.
Since $R_j^{a,b,N}=0$ for $j<0$, we deduce that
$$
\Phi_{0}^{a,b,N;\M,U}=(-1)^{n_ua} \Phi_{0}^{a,b,N;\M} \Lambda\prod_{j=1}^a(b+N+j+1)_{a-j},
$$
and the Lemma \ref{lem2} gives  $\Phi_{0}^{a,b,N;\M,U}\not =0$.

The  proof can now be completed as that of Step 3 in Theorem \ref{th1}.

\end{proof}

Multiple roots of the polynomial $\prod_{u\in U}(x-\lambda^{a,b}(u))$ can be managed using derivatives of $R_{n+n_U-j+1}^{a,b,N}(\lambda^{a,b}(u))$ in the determinant (\ref{qusmei2u}).

We next explicitly compute the three term recurrence formula for the orthogonal polynomials $(q_n^{a,b,N;\M,U})_n$ (it will be useful in  \cite{dundh}). We define
\begin{equation}\label{deflau}
\Lambda_n^{a,b,N;\M,U}=
\left|
  \begin{array}{@{}c@{}lccc@{}c@{}}
  &  &&\hspace{-3cm}{}_{1\le j\le a+n_U+1,j\not=2} \\
    \dosfilas{(b+N-n-n_U+j)_{a+n_U+1-j}W_{g}^{a,b,N;\M}(-n-n_U+j-2) }{g\in \{b,b+1,\cdots, a+b-1\}}\\
    \dosfilas{(-1)^{j-1}R_{n+n_U-j+1}^{a,b,N}(\lambda^{a,b}(u))}{u\in U}
  \end{array}\hspace{-.6cm}\right| .
\end{equation}
\begin{corollary}\label{cor5.2}
In the hypothesis of Theorem \ref{th2}, the  orthogonal polynomials $(q_n^{a,b,N;\M,U})_n$ satisfy the following three term recurrence formula
$$
xq_n^{a,b,N;\M,U}=a_{n+1}q_n^{a,b,N;\M,U}+b_nq_n^{a,b,N;\M,U}+c_nq_n^{a,b,N;\M,U},
$$
where
\begin{align*}
a_n&=(n+n_U)\frac{\Phi_{n-1}^{a,b,N;\M,U}}{\Phi_n^{a,b,N;\M,U}},\\
b_n&=(n+n_U)(b+N-n-n_U+1)+(a+n+n_U+1)(N-n-n_U)\\&\hspace{2cm}-\Delta\left((n+n_U)\frac{\Lambda_n^{a,b,N;\M,U}}{\Phi_n^{a,b,N;\M,U}} \right),\\
c_n&=n(a+N-n+1)(a+b+N-n+1)\left(\frac{a+b+N-n}{a+b+N-n+1}\right)^{a}\\&\hspace{4cm}\times \frac{\Phi_{n+1}^{a,b,N;\M,U}}{\Phi_n^{a,b,N;\M,U}},
\end{align*}
where $\Delta$ denotes the first order difference operator $\Delta f=f(n+1)-f(n)$.
\end{corollary}

\begin{proof}
It is a matter of computation using the formulas (\ref{lcqu}) and (\ref{normqu}) for the leading coefficient and the norm of $q_n^{a,b,N;\M,U}$, respectively.
\end{proof}

We complete this Section with a couple of remarks.

\medskip

\begin{remark}
Note that only rational functions of $N$ appear in the three term recurrence formula in Corollary \ref{cor5.2}. Using standard analyticity arguments, we deduce that
the three term recurrence formula is also true for $N\in \CC$ except for the poles of $\Phi_n^{a,b,N;\M,U}$ (as functions of $N$).
\end{remark}

\medskip

\begin{remark}\label{bbn}
Lemma \ref{lem2} allows us to extend Theorems \ref{th1} and \ref{th2} for other Christoffel transforms of the measure $\nu_{a,b,N}^\M$. We just sketch the idea. Indeed, it is easy to see that we can substitute in the identities (\ref{lasff}) and (\ref{lasff2}) the power $x^m$ for any polynomial $r$ of degree $m$, changing the number $u_i^m$ (\ref{losuim}) to
\begin{equation}\label{lanup}
u_i^r=\begin{cases} \frac{r'(\lambda^{a,b} (-i-1))P_i(-(\lambda^{a,b} (-i-1))}{P_i'(-(\lambda^{a,b} (-i-1))},& \mbox{if $i$ satisfies  (\ref{lascr}),}\\
0,& \mbox{otherwise,}\end{cases}
\end{equation}
and the sequences $(\psi^m_g)_m$ to
\begin{equation}\label{lanpsp}
\psi^r_g=\begin{cases}\frac{\left(r(\lambda^{a,b} (-g-1))+u_g^r\right)\operatorname{Res}_{\lambda^{a,b} (-g-1)}(1/P)}{W_g^{a,b,N;\M}(0)}, &g=b,\cdots,\lceil\frac{a+b}{2}\rceil -1, \\\vspace{.1cm}
\frac{r(\lambda^{a,b} (-g-1))}{P_g(-(\lambda^{a,b} (-g-1))h_g^{-a,-b,-2-N}(0)}, &g=\lceil\frac{a+b}{2}\rceil ,\cdots,a-1,\\\vspace{.1cm}
\frac{r(\lambda^{a,b} (-g-1)) \operatorname{Res}_{\lambda^{a,b} (-g-1)}(1/P)}{W_g^{a,b,N;\M}(0)}, &g=a,\cdots,a+b-1.
\end{cases}
\end{equation}
Consider now a set $G$, $G\subset \{b,\cdots, a+b-1\}$, define the finite set of positive integers
$$
H_G=\{b,\cdots, a+b-1\}\setminus G,
$$
the polynomial
$$
s_G(x)=\prod_{h\in {H_G}}(\lambda^{a,b}(x)-\lambda^{a,b}(-h-1)),
$$
and assume that if $h\in H_G$ and $b\le h\le \lceil\frac{a+b}{2}\rceil -1$, the multiplicity of $-h-1$ as a root of $s_G$ is bigger than $1$.
Then the polynomials
\begin{equation}\label{oqpm}
q_n^{\M,G}(x)=
\left|
  \begin{array}{@{}c@{}lccc@{}c@{}}
  & (-1)^{j-1}R_{n-j+1}^{a,b,N}(x) &&\hspace{-2.6cm}{}_{1\le j\le n_G+1} \\
    \dosfilas{(b+N-n+j)_{a+1-j}W_{g}^{a,b,N;\M}(-n+j-2) }{g\in G}\\
     \end{array}\hspace{-.8cm}\right|,
\end{equation}
are orthogonal with respect to the measure
\begin{equation}\label{lanmp}
\prod_{h\in {H_G}}(x-\lambda^{a,b}(-h-1))\nu_{a,b,N}^\M.
\end{equation}
Indeed, if we write $r(x)=s_G(x)x^m$, it is easy to see that $\psi_g^r=0$, $g\not \in G$, $g=b,\cdots, a+b-1$. Using then the version of Lemma \ref{lem2} provided by (\ref{lanup}) and (\ref{lanpsp}), the orthogonality of
the polynomials $q_n^{\M,G}$ with respect to the measure (\ref{lanmp}) can be proved proceeding as in the proof of Theorem \ref{th1}.

If we set $U_G=\{-h-1; h\in H_G\}$, this result is actually saying that (\ref{oqpm}) provides for the orthogonal polynomials with respect to the measure
(\ref{lanmp}) other determinantal expression different to (\ref{qusmei2u}). However, there is a little improvement because $U_G$ has not to satisfy (\ref{hipct}). For instance, for $a=5$, $b=2$ and $G=\{3\}$, we get $H_G=\{2,4,5,6\}$ and $U_G=\{-7,-6,-5,-3\}$. $H_G$ does not satisfy (\ref{hipct}) but (\ref{oqpm}) provided orthogonal polynomials with respect to the measure (\ref{lanmp}).

We can also proceed as in Theorem \ref{th2}. Indeed, consider a finite set $U$ of complex number satisfying (\ref{hipct}). We can then prove as in Theorem \ref{th2} that the polynomials
\begin{equation*}\label{oqpmu}
q_n^{\M,G,U}(x)=
\frac{\left|
  \begin{array}{@{}c@{}lccc@{}c@{}}
  & (-1)^{j-1}R_{n+n_U-j+1}^{a,b,N}(x) &&\hspace{-2.6cm}{}_{1\le j\le n_G+n_U+1} \\
    \dosfilas{(b+N-n-n_U+j)_{a+n_U+1-j}W_{g}^{a,b,N;\M}(-n-n_U+j-2) }{g\in G}\\
    \dosfilas{(-1)^{j-1}R_{n+n_U-j+1}^{a,b,N}(\lambda^{a,b}(u))}{u\in U}
  \end{array}\hspace{-.5cm}\right|}{\prod_{u\in U}(x-\lambda^{a,b}(u))}
\end{equation*}
are orthogonal with respect to the measure
$$
\nu_{a,b,N}^{\M,G,U}=\prod_{u\in U}(x-\lambda^{a,b}(u))\prod_{h\in H_G}(x-\lambda^{a,b}(-h-1))\nu_{a,b,N}^{\M}
$$
providing that
$$
\Phi_n^{\M,G,U}=
\frac{\left|
  \begin{array}{@{}c@{}lccc@{}c@{}}
  &  &&\hspace{-2.6cm}{}_{1\le j\le n_G+n_U} \\
    \dosfilas{(b+N-n-n_U+j+1)_{a+n_U-j}W_{g}^{a,b,N;\M}(-n-n_U+j-1) }{g\in G}\\
    \dosfilas{(-1)^{j}R_{n+n_U-j}^{a,b,N}(\lambda^{a,b}(u))}{u\in U}
  \end{array}\hspace{-.6cm}\right|}{\prod_{u\in U}(x-\lambda^{a,b}(u))}\not=0.
$$
The norm is then given by
$$
\frac{(-1)^{n_U}n!(n+1)_{a-n_G}(N+b)!^2(N+a+b-n)^{n_G}}
{(n+n_U)!(N+n_G-n)!(N+n_G+b-n)!}
\Phi_n^{\M,G,U}\Phi_{n+1}^{\M,G,U}.
$$
There is again an improvement on Theorem \ref{th2}, because now there can be double roots in the polynomial
\begin{equation}\label{pdlc}
\prod_{u\in U}(x-\lambda^{a,b}(u))\prod_{h\in H_G}(x-\lambda^{a,b}(-h-1)).
\end{equation}
That is the case, for instance, when $a=b=2$, $G=\{3\}$ and $U=\{-3\}$. Indeed, since $H_G=\{2\}$, the polynomial (\ref{pdlc}) is then equal to $(x-\lambda^{2,2}(-3))^2$.

\end{remark}

\section{New Krall dual Hahn measures}\label{sec4}
The more interesting case of the Christoffel transforms  studied in the previous Section is when
\begin{equation}\label{lacsu}
\lambda^{a,b}(u)\in \{\lambda^{a,b}(i):\mbox{$-b\le i\le -2$ or $0\le i\le N$}\},\quad u\in U,
\end{equation}
because then the measure $\nu_{a,b,N}^{\M,U}$ (\ref{laxm}) is a Krall measure.

Since $\lambda^{a,b}(u)=\lambda^{a,b}(-u-a-b-1)$, and we still have to assume (\ref{hipct}), we can take
\begin{equation}\label{lacsu2}
U\subset \{i:-a-b+1\le i\le -a-1\}\cup \{i:1\le i\le N\}.
\end{equation}
The measure $\nu_{a,b,N}^{\M,U}$ is supported in the finite sets of integers
$$
\{\lambda^{a,b}(i):i\in \{-b,\cdots, N\}\setminus U\}
$$
which has $b+N+1-n_U$ elements.

If we write $n_-$ for the number of elements of $\{u\in U: -a-b+1\le u\le -a-1\}$, the measure $\nu_{a,b,N}^{\M,U}$ dependes on the $b-n_-$ continuous parameters
$M_{j}$ with $-b\le j\le -1$ and $j\not \in U$.

In this Section, we prove that under the assumption (\ref{lacsu}), the orthogonal polynomials with respect to the measure $\nu_{a,b,N}^{\M,U}$ are eigenfunctions of a higher order difference operator of the form (\ref{hodo}). We prove this by constructing for these polynomials other determinantal formula (different to (\ref{qusmei2u})).

We need to introduce some notation. Define the numbers
\begin{align}\label{lospu1}
a_U&=a+\max(-1,\max U)+1,\quad b_U=b+\max(-1,\max U)+1\\\label{lospu2}
N_U&=N-\max(-1,\max U)-1,\quad \hspace{-.12cm}s_U=\lambda^{a,b}(\max(-1,\max U)+1),
\end{align}
and the finite sets of positive integers
\begin{align}\label{thesetf}
F_U&=\{a,\cdots, a+b-1\}\cup \{a+b+u:u\in U\},\\\label{thesetg}
G_U&=I(F_U),
\end{align}
where $I$ is the involution defined by (\ref{dinv}). The elements of $G_U$ are arranged in increasing order.
Consider finally the sequence
\begin{equation}\label{pqusmei2k}
\tilde \Phi_n^{a,b,N;\M,U}(x)=
\left|
  \begin{array}{@{}c@{}lccc@{}c@{}}
  &  &&\hspace{-.6cm}{}_{1\le j\le n_{G_U}} \\
    \dosfilas{W_{g}^{a_U,b_U,N_U;\M}(-n+j-1) }{g\in G_U}
  \end{array}\hspace{-.3cm}\right|.
\end{equation}

\begin{theorem}\label{th3} Let $a,b,N$ be nonnegative integers with $1\le b\le a\le N$,  write $\M=\{M_0,\cdots, M_{b-1}\}$ for a finite set consisting of $b$ real parameters, $M_i\not =0,1$, and $U$ for a finite set of integers satisfying (\ref{lacsu}).
Then the measure $\nu_{a,b,N}^{\M,U}$ has a sequence of orthogonal polynomials if and only if
\begin{equation}\label{cnysuk}
\tilde\Phi_n^{a,b,N;\M,U}(n)\not =0,\quad n=0,\cdots, N+b-n_U+1.
\end{equation}
In that case the sequence of polynomials $(\tilde q_n^{a,b,N;\M,U})_{n=0}^{N+b-n_U}$ defined by
\begin{equation}\label{qusmei2k}
\tilde q_n^{a,b,N;\M,U}(x)=
\left|
  \begin{array}{@{}c@{}lccc@{}c@{}}
  & \frac{(-1)^{j-1}R_{n-j+1}^{a_U,b_U,N_U}(x-s_U)}{(b+N-n+j)_{n_{G_U}+1-j}} &&\hspace{-.6cm}{}_{1\le j\le n_{G_U}+1} \\
    \dosfilas{W_{g}^{a_U,b_U,N_U;\M}(-n+j-2) }{g\in G_U}
  \end{array}\hspace{-.3cm}\right|,
\end{equation}
is orthogonal  with respect to the measure $\nu_{a,b,N}^{\M,U}$, with norm
\begin{align}\label{normqk}
&\langle \tilde q_n^{a,b,N;\M,U},\tilde q_n^{a,b,N;\M,U}\rangle_{\nu_{a,b,N}^{\M,U}}=\tilde\Phi_n^{a,b,N;\M,U}\tilde\Phi_{n+1}^{a,b,N;\M,U}\\\nonumber&\hspace{3cm}\times\frac{(n+n_U)!(N+b)!^2(N+b-n)!}
{n!(N+a-n-n_U)!(N+b-n+n_{G_U})!^2}.
\end{align}
Moreover the polynomials $\tilde q_n^{a,b,N;\M,U}(\lambda^{a_U,b_U}(x))$, $n\ge 0$, are eigenfunctions of a higher order difference operator of the form (\ref{hodo}) with
$$
-s=r=\sum_{f\in F_U}f-\binom{n_{F_U}}{2}+1.
$$
\end{theorem}

First of all, we explain how we have found the formula (\ref{qusmei2k}) for the orthogonal polynomials with respect to the measure $\nu_{a,b,N}^{\M,U}$. For $s$ small enough, write (in a similar form to (\ref{losas}) in Section \ref{sec1})
$$
a_{U,s}=a_U-s/M,\quad b_{U,s}=b_U+s
$$
so that $a_{U,s}, b_{U,s}\not \in \ZZ$. Consider the measure $\rho ^{F_U}_{a_{U,s}, b_{U,s},N_U}$ defined by (\ref{mqs}). When all the parameters in $\M$ are equal, that is, $\M=\{M,\cdots, M\}$, a careful computation shows that the measure $\rho ^{F_U}_{a_{U,s}, b_{U,s},N_U}$ converges to $\nu_{a,b,N}^{\M,U}$ as $s\to 0$.
Since the determinantal formula (\ref{qusmeix}) provides orthogonal polynomials with respect to $\rho ^{F_U}_{a_{U,s}, b_{U,s},N_U}$, we can then construct orthogonal polynomials with respect to $\nu_{a,b,N}^{\M,U}$ by taking limits in (\ref{qusmeix}) as $s\to 0$. As explained in Section \ref{sec1}, in order to avoid the collapse of the determinant (\ref{qusmeix}) when passing to the limit, we have to change the Hahn polynomials $h_g^{-a_U,-b_U,-2-N_U}$ by the polynomials $W_g^{a_U,b_U,N_U;\M}$ (\ref{losw}). In doing that we get the determinantal formula (\ref{qusmei2k}).

\begin{proof}[Proof of Theorem \ref{th3}]
We first prove that the polynomials $\tilde q_n^{a,b,N;\M,U}(\lambda^{a_U,b_U}(x))$, $n\ge 0$, are eigenfunctions of a higher order difference operator of the form (\ref{hodo}). As in Theorem \ref{th1}, this is a consequence of the determinantal formula (\ref{qusmei2k}). Indeed, write $p_n(x)=R_n^{a_U,b_U,N_U}(\lambda^{a_U,b_U}(x))$ and $\A$ for the algebra $\A^{\lambda}$ (\ref{defal}) with $\lambda(x)=x(x+a_U+b_U+1)$. Since $b_U+N_U-n+1=b+N-n+1$, we have again that $\epsilon_n=b+N-n+1$ defines a $\D$-operator for the sequence of Hahn polynomials $(R_n^{a_U,b_U,N_U})_n$ and the algebra $\A$, and $\xi_{n,i}=(b+N-n+1)_i$ (see Step 4 in the proof of Theorem \ref{th1}).
Hence, the polynomials $(\tilde q_n^{a_U,b_U,N_U;\M,U}(\lambda^{a_U,b_U}(x)))_n$ have the form (\ref{qusmeig}) and they are eigenfunctions of an operator of the form (\ref{hodo}). The order can be computed as in  \cite[Theorem 3.1]{dudh}.

The rest of the proof of Theorem \ref{th3} can be done as that of Theorem \ref{th1} but using the following version of Lemma \ref{lem2}.

We start by extending the definitions previous to the Lemma \ref{lem2} to the new scenario. In order to do that, we have to take into account that the parameters $a_U,b_U,N_U$ ((\ref{lospu1}) and (\ref{lospu2})) play in the polynomials $\tilde q_n^{a_U,b_U,N_U;\M,U}$ the role played by the parameters $a,b,N$ in the polynomials $q_n^{a,b,N;\M}$. In the same way, the set $G_U$ plays now the role of $\{b,\cdots, a+b-1\}$. It is not difficult to check that $\{b_U,b_U+1,\cdots, a_U-1\}\subset G_U$. Hence, we introduce the following auxiliary functions.

We define the polynomial $P_U$ as follows (compare with (\ref{elpp}))
\begin{equation*}\label{elppu}
P_U(x)=\prod_{g\in G_U}(x+\lambda^{a_U,b_U} (-g-1)).
\end{equation*}
It is not difficult to see that if $i=b_U,\cdots, a_U-1$, then $-\lambda^{a_U,b_U} (-i-1)$ is a double root of $P_U$.
Define then the polynomial $P_{U,i}$ by (compare with (\ref{elppi}))
\begin{equation*}\label{elppiu}
P_{U,i}(x)=\frac{(2i+1-a_U-b_U)P_U(x)}{(x+\lambda^{a_U,b_U} (-i-1))^2},
\end{equation*}
Since $\lambda^{a_U,b_U} (-i-1)=\lambda^{a_U,b_U} (i-a_U-b_U)$, we get that
$P_{U,i}=-P_{U,a_U+b_U-1-i}$ when
\begin{align}\label{lascru}
&\mbox{either $i=b_U,\cdots, \lceil\frac{a_U+b_U}{2}\rceil-2$}\\\nonumber &\hspace{2.5cm}\mbox{or $i=\lceil\frac{a_U+b_U}{2}\rceil-1$ and $a_U+b_U=2\lceil\frac{a_U+b_U}{2}\rceil$.}
\end{align}
Define also the numbers $u_{U,i}^m$, $m\ge 0$, by (compare with (\ref{losuim}))
\begin{equation*}\label{losuimu}
u_{U,i}^m=\begin{cases}\frac{m(\lambda^{a_U,b_U} (-i-1))^{m-1} P_{U,i}(-\lambda^{a_U,b_U} (-i-1))}{P_{U,i}'(-\lambda^{a_U,b_U} (-i-1))},& \mbox{if $i$ satisfies  (\ref{lascru}),}\\
0,& \mbox{otherwise.}\end{cases}
\end{equation*}
For $g\in G_U$, we finally define the sequences $(\psi^m_{U,g})_m$ as follows (compare with (\ref{lasphi}))
\begin{equation*}\label{lasphiu}
\psi^m_{U,g}=\begin{cases}\frac{\left((\lambda^{a_U,b_U} (-g-1))^m +u_{U,g}^m\right)\operatorname{Res}_{-\lambda^{a_U,b_U} (-g-1)}(1/P_U)}{W_g^{a_U,b_U,N_U;\M}(0)}, &g=b_U,\cdots,\lceil\frac{a_U+b_U}{2}\rceil -1, \\
\frac{(\lambda^{a_U,b_U} (-g-1))^m}{P_{U,g}(-\lambda^{a_U,b_U} (-g-1))h_g^{-a_U,-b_U,-2-N_U}(0)}, &g=\lceil\frac{a_U+b_U}{2}\rceil ,\cdots,a_U-1,\\
\frac{(\lambda^{a_U,b_U} (-g-1))^m \operatorname{Res}_{-\lambda^{a_U,b_U} (-g-1)}(1/P_U)}{W_g^{a_U,b_U,N_U;\M}(0)}, &\mbox{otherwise}.
\end{cases}
\end{equation*}
The following identities can be proved as (\ref{lasff}) and (\ref{lasff2}) in Lemma \ref{lem2}.
For $0\le n$, $0\le m\le n$ and $m-a+1\le s\le n$
\begin{align*}
&\frac{(N_U+1)!\langle R_{s}^{a_U,b_U,N_U}(x-s_U),(x-s_U)^m\rangle _{\nu_{a,b,N}^{\M,U}}}{(-1)^{n_{G_U}+s+1}(a_U-1)!(N+b)!}\\
&\hspace{4cm} =(b+N-s+1)_{s}\sum_{g\in G}\psi^m_{U,g} W_g^{a_U,b_U,N_U;\M}(-s-1),
\end{align*}
and for $n=0,1,\cdots, a+N-n_{U}+1$
\begin{align*}
&\frac{(-1)^{n+1}(N_U+1)!\langle R_{n-n_{G_U}}^{a_U,b_U,N_U}(x-s_U),(x-s_U)^n\rangle _{\nu_{a,b,N}^{\M,U}}}{(a_U-1)!(N+b)!(b+N-n+n_{G_U}+1)_{n-n_{G_U}}}\\\nonumber&\hspace{.7cm}
=\frac{(n+n_U)!(N+a-n-n_U+1)_n}{(-1)^{n+1}(a_U-1)!(N_U+2)_{a_U-1-n_U}}+\sum_{g\in G}\psi^m_{U,g} W_g^{a_U,b_U,N_U;\M}(-n+n_{G_U}-1).
\end{align*}

\end{proof}

\section{More pieces of the puzzle}\label{sec5}
When the finite set $U$ satisfies (\ref{lacsu}), we have found two nontrivial determinantal representations for the orthogonal polynomials
with respect to the Krall discrete measure $\nu_{a,b,N}^{\M,U}$.

We show in this Section  that actually this is also the case of the families of the orthogonal polynomials
with respect to all the Christoffel transforms studied in Section \ref{sec3}.
In particular, this includes other determinantal expression for
a sequence of orthogonal polynomials with respect to the basic example $\nu_{a,b,N}^{\M}$.

\begin{theorem}\label{th4} Let $a,b,N$ be nonnegative integers with $1\le b\le a\le N$. For a finite set $\M=\{M_0,\cdots, M_{b-1}\}$  consisting of $b$ real parameters, $M_i\not =0,1$, write $\M^{-1}$ for the set of parameters $\M^{-1}=\{1/M_0,\cdots , 1/M_{b-1}\}$. Let $U$ be a finite set of complex numbers satisfying (\ref{hipct}) and $n_S>0$ and  define the sequence $(\Psi_n^{a,b,N;\M,U})_n$ by
\begin{equation}\label{defphiur}
\Psi_n^{a,b,N;\M,U}=
\left|
  \begin{array}{@{}c@{}lccc@{}c@{}}
  & &&\hspace{-1.9cm}{}_{1\le j\le b+n_U} \\
    \dosfilas{(-N-a-b)_{n+j-1}W_{f}^{a,b,-2-N-a-b;\M^{-1}}(N+a+b-n-j+1) }{f\in \{a,a+1,\cdots, a+b-1\}}\\
     \dosfilas{R_{n-b+j-1}^{b,a,N}(\lambda^{a,b}(u)) }{u\in U}
  \end{array}\hspace{-.5cm}\right| .
\end{equation}
Then the measure $\nu_{a,b,N}^{\M,U}$ has a sequence of orthogonal polynomials if and only if
\begin{equation}\label{cnysur}
\Psi_n^{a,b,N;\M,U}(n)\not =0,\quad n=0,\cdots, n_S.
\end{equation}
In that case the sequence of polynomials $(r_n^{a,b,N;\M,U})_{n=0}^{n_S-1}$ defined by
\begin{equation}\label{qusmer}
r_n^{a,b,N;\M,U}(x)=
\frac{\left|
  \begin{array}{@{}c@{}lccc@{}c@{}}
  & R_{n-b+j-1}^{b,a,N}(x) &&\hspace{-2cm}{}_{1\le j\le b+n_U+1} \\
    \dosfilas{(-N-a-b)_{n+j-1}W_{f}^{a,b,-2-N-a-b;\M^{-1}}(N+a+b-n-j+1) }{f\in \{a,a+1,\cdots, a+b-1\}}\\
     \dosfilas{R_{n-b+j-1}^{b,a,N}(\lambda^{a,b}(u)) }{u\in U}
  \end{array}\hspace{-.5cm}\right|}{\prod_{u\in U}(x-\lambda^{a,b}(u))},
\end{equation}
is orthogonal  with respect to the measure $\nu_{a,b,N}^{\M,U}$, with norm
\begin{equation}\label{normqy}
\langle r_n^{a,b,N;\M,U},r_n^{a,b,N;\M,U}\rangle_{\nu_{a,b,N}^{\M,U}}=\frac{n!(-N-a-b)_n^2(N+b+1-n)_a\Psi_n^{a,b,N;\M,U}\Psi_{n+1}^{a,b,N;\M,U}} {(-1)^{n_U+b}(n+n_U)!(N+b+1)_a^2}.
\end{equation}
\end{theorem}

Before going with the proof, we analyze the case $U=\emptyset$, that is, the basic example $\nu_{a,b,N}^{\M}$.
As in the previous Section, we explain how we have found the formula (\ref{qusmer}) for the orthogonal polynomials with respect to the measure $\nu_{a,b,N}^{\M}$. For $F_0=\{a,\cdots, a+b-1\}$, and for $s$ small enough, consider the numbers
$$
a_{s}=a-s/M,\quad b_{s}=b_U+s,\quad \hat a_s=-b-s/M,\quad \hat b_s=-a+s,\quad \hat N=N+a+b
$$
(see (\ref{losas}) and (\ref{hats}) in Section \ref{sec1}).
Consider the measure $\rho ^{F_0}_{a_{s}, b_{s},\hat N}$ defined by (\ref{mqs}). Since this measure is a Christoffel transform of the dual Hahn measure $\rho_{\hat a_s,\hat b_s,\hat N}$ (which it is well defined because $a_{s}, b_{s}\not \in \ZZ$), we can construct a sequence of orthogonal polynomials with respect to $\rho ^{F_0}_{a_{s}, b_{s},\hat N}$ by mean of the formula (see \cite[Theorem 2.5]{Sz}):
\begin{equation}\label{qusmerr}
p_n(x)=
\frac{\left|
  \begin{array}{@{}c@{}lccc@{}c@{}}
  & R_{n+j-1}^{\hat a_s,\hat b_s,\hat N}(x+a+b) &&\hspace{-0.7cm}{}_{1\le j\le b+1} \\
        \dosfilas{R_{n+j-1}^{\hat a_s,\hat b_s,\hat N}(\lambda^{\hat a_s,\hat b_s}(f)) }{i\in \{a,\cdots , a+b-1\}}
  \end{array}\hspace{-.3cm}\right|}{\prod_{i=1}^b(x+a+b-\lambda^{\hat a_s,\hat b_s}(f)))}.
\end{equation}
As explained in Section \ref{sec1}, when all the parameters in $\M$ are equal, that is, $\M=\{M,\cdots, M\}$,
the measure $\rho ^{F_0}_{a_{s}, b_{s},\hat N}$ converges to $\nu_{a,b,N}^{\M}$ as $s\to 0$.
Since the determinantal formula (\ref{qusmerr}) provides orthogonal polynomials with respect to $\rho ^{F_0}_{a_{s}, b_{s},\hat N}$, we can then construct orthogonal polynomials with respect to $\nu_{a,b,N}^{\M}$ by taking limits in (\ref{qusmerr}) as $s\to 0$.

On the one hand, a careful computation using the duality (\ref{sdm2b}) and the identity (\ref{hcp}) shows that for $b\le n$ and $a\le f\le a+b-1$
\begin{align}\label{pmu1}
&\lim_{s\to 0}\frac{1}{s}R_{n}^{\hat a_s,\hat b_s,\hat N}(\lambda^{\hat a_s,\hat b_s}(f))\\\nonumber &\hspace{1cm}=\frac{(1-M)(-N-a-b)_n W_f^{a,b,-2-N-a-b;\M^{-1}}(N+a+b-n)}{(-1)^fM(n-b+1)_b(f-b)!(-N-a-b)_f}.
\end{align}
On the other hand, using the identity (\ref{dhpn}), we have
\begin{equation}\label{pmu2}
\frac{R_{n}^{-b,-a,N+a+b}(x+a+b)}{\prod_{f=a}^{a+b-1}(x+a+b-\lambda^{-a,-b}(f))}=\frac{R_{n-b}^{b,a,N}(x)}{(n-b+1)_b}.
\end{equation}
Using (\ref{pmu1}) and (\ref{pmu2}) we see that for $\M=\{M,\cdots, M\}$ and $U=\emptyset$, the limit of the polynomials (\ref{qusmerr}) are the polynomials (\ref{qusmer}) (after renormalization). This is the way we have found (\ref{qusmer}).

\begin{proof}
We first consider the  basic example $\nu_{a,b,N}^{\M}$, i.e., $U=\emptyset$.

The key is again some identities of the kind displayed in Lemma \ref{lem2}. More precisely:
define the polynomial $Q$ as follows
\begin{equation}\label{elqq}
Q(x)=\prod_{j=a}^{a+b-1}(x-a-b-\lambda^{a,b} (-j-1)).
\end{equation}
Note that the roots of $Q$ are simple.

For $f=a,\cdots, a+b-1$, we next define the sequences $(\psi^m_f)_m$ as follows
\begin{equation}\label{lasphiy}
\psi^m_f=\frac{(a+b+\lambda^{a,b} (-f-1))^m}{Q'(a+b+\lambda^{a,b} (-f-1))W_f^{a,b,-2-N-a-b;\M^{-1}}(a+N+1)}.
\end{equation}
We then have for $0\le n\le b+N$,  $0\le m\le n$ and $n-b+1\le s$
\begin{align}\label{lasffy}
&\langle R_{s}^{b,a,N}(x),(x+a+b)^m\rangle _{\nu_{a,b,N}^\M}=\frac{(b-1)!(N+2)_{b-1}(-a-b-N)_{s+b}}{(-1)^{b}(b+N+1)_a}\\\nonumber
&\hspace{.8cm}\times \sum_{f=a}^{a+b-1}\psi^m_f W_f^{a,b,-2-N-a-b;\M^{-1}}(a+N-s),
\end{align}
and for $0\le n\le N+b$
\begin{align}\label{lasff2y}
&\langle R_{n-b}^{b,a,N},(x+a+b)^n\rangle _{\nu_{a,b,N}^\M}
=\frac{n!(b+N+1-n)_a(-a-b-N)_n^2}{(b+N+1)_a^2}\\\nonumber&\hspace{.3cm}+\frac{(b-1)!(N+2)_{b-1}(-a-b-N)_{n}}{(-1)^{b}(b+N+1)_a}
\sum_{f=a}^{a+b-1}\psi^n_f W_f^{a,b,-2-N-a-b;\M^{-1}}(a+b+N-n).
\end{align}
As mentioned in the proof of Lemma \ref{lem2}, these kind of identities appear in all the families of Krall-discrete polynomials (see \cite[p. 69, 77]{DdI}, \cite[p. 380-381]{DdI2} for the Krall Charlier, Krall Meixner and Krall Hahn polynomials, respectively). The identities (\ref{lasffy}) and (\ref{lasff2y}) are completely similar to the identities (\ref{lasid}) and (\ref{lasid2}) for the Dual Hahn polynomials and the measure $\rho_{a,b,N}^F$ (\ref{mqs}) when the finite set $F$ satisfies $a,b\ge \max F+1$. Indeed, on the one hand, all the roots of the polynomial $Q$ (\ref{elqq}) are simple as those of the polynomial (\ref{elp}) (compare with the situation in Lemma \ref{lem2} explained in  Remark \ref{urm}). And, on the other hand, proceeding as in Remark \ref{urm2}, one can see that the polynomials $W_f^{a,b,-2-N-a-b;\M^{-1}}(a+N-x)$, $a\le f\le a+b-1$, which appear in the right hand side of the identities (\ref{lasffy}) and (\ref{lasff2y}), are eigenfunctions of the second order difference operator
\begin{equation*}\label{elsoph2}
D=A(x)\Sh_{-1}+B(x)\Sh_0+C(x)\Sh_{1},
\end{equation*}
where
\begin{align*}
A(x)&=(x+1)(x-a-N),\quad C(x)=(x-N-1)(x+b), \\\nonumber
B(x)&=-A(x-1)-C(x+1),
\end{align*}
and
$$
D(W_f^{a,b,-2-N-a-b;\M^{-1}}(a+N-x))=\lambda^{a,b}(-f-1)W_f^{a,b,-2-N-a-b;\M^{-1}}(a+N-x)
$$
(note that the eigenvalues $\lambda^{a,b}(-f-1)$ define the polynomial $Q$ (\ref{elqq})).
$D$ is the same second order difference operator with respect to which the Hahn polynomials $h_f^{-a,-b,N+a+b}(a+N-x)$, $0\le f$, are eigenfunctions (see (\ref{defbc})) (compare with the situation in Lemma \ref{lem2} explained in the Remarks \ref{urm2}).

Taking this into account, the identities (\ref{lasffy}) and (\ref{lasff2y}) can be proved in a similar way to the identities (\ref{lasid}) and (\ref{lasid2}).

The basic example $\nu_{a,b,N}^{\M}$ (i.e., $U=\emptyset$) can now be proved as Theorem \ref{th1}, and  the general case when $U\not =\emptyset$  can be proved as Theorem \ref{th2}.
\end{proof}

We finish this Section pointing out that we have constructed three different determinantal formulas for the orthogonal polynomials with respect to the Krall dual Hahn measure $\nu_{a,b,N}^{\M,U}$ when $U$ satisfies (\ref{lacsu}). The first one is (\ref{qusmei2u}) whose determinant has size $a+n_U+1$. The second one is (\ref{qusmei2k}) whose determinant has size $\max\{a+b-1,a+b+U\}-n_{U}+1$ (for the computation of the size we have used (\ref{spmu})). The third one is (\ref{qusmer}) whose determinant has size $b+n_U+1$. Note, that the size of those determinants can be very different. For instance, for $a=5,b=2$ and $U=\{-2,0,1,5,6\}$, the size of the three determinants are $11$, $9$ and $8$, respectively. None of these three determinants can be transformed in some of the other determinants by combining rows
and columns.
\bigskip

\section{The case $a\le b$}
In the previous Sections, we have assumed that $b\le a$. As far as we know, the dual Hahn polynomials do not have any symmetry between the parameters $a$ and $b$, so the case $a\le b$ needs some specific changes to be handled. Those changes are however rather natural: only the basic measure
$\nu_{a,b,N}^\M$ and the polynomials $(W_g^{a,b,N;\M})_g$ need  to be slightly adapted.

When $a\le b$, the set of real parameters is now $\M=\{M_0,\cdots, M_{a-1}\}$, $M_i\not =0,1$, and its number of elements is $a$.
The basic case corresponds with the discrete measure $\nu_{a,b,N}^\M$ supported in the finite quadratic net
$$
\{\lambda^{a,b}(i):i=-a,\cdots ,N\}
$$
and defined by
\begin{align}\label{lanuf}
\nu_{a,b,N}^{\M}=&\sum_{x=-a}^{-1}\frac{(2x+a+b+1)(N+1-x)_{x+b}}{(N+b+1)_{x+a+1}}M_{x+a}\delta_{\lambda^{a,b}(x)}\\\nonumber
&\qquad +\frac{(N+1)_b^2}{(b+1)_{a-b}}\sum_{x=0}^N \frac{\rho_{b,a,N}(x)}{\prod_{i=0}^{b-1}(x+a+i+1)(x+b-i)}\delta _{\lambda^{a,b}(x)},
\end{align}
where $\rho_{b,a,N}$ is the dual Hahn measure (see (\ref{masdh}) above).

We next define the polynomials $(W_g^{a,b,N;\M})_g$.

For $g\in\{\lceil\frac{a+b}{2}\rceil,\cdots ,b-1\}$, we use again the limit (\ref{lim1c}) and define
$$
W_g^{a,b,N;\M}(x)=\lim_{s\to 0}\frac{1}{s}\left(h_{g}^{-a-s,-b,-2-N}(x)-\frac{h_{g}^{-a-s,-b,-2-N}(-2-N)}{h_{-g+a+b-1}^{-a-s,-b,-2-N}(-2-N)}h_{-g+a+b-1}^{-a-s,-b,-2-N}(x)\right).
$$
The reason why we have substituted $0$ by $-2-N$ in the previous limit with respect to (\ref{lim1c}) is to preserve
the symmetry of the Hahn polynomials with respect to the interchange of the parameters $a$ and $b$. Indeed, it is easy to see that
\begin{equation}\label{fi2a}
(-1)^nW_{g}^{a,b,N;\M}(x)=W_{g}^{b,a,N;\M^{-1}}(-2-N-x).
\end{equation}

For $g\not \in\{\lceil\frac{a+b}{2}\rceil,\cdots ,b-1\}$, we define $W_g^{a,b,N;\M}$ as follows
\begin{equation}\label{loswf}
\begin{cases}
 (-1)^{b+g}(g-b)!&\\
\hspace{.2cm}\times \left[(b+a-g-1)!(-x)_ah_{g-a}^{a,-b,-2-N-a}(x-a)\right.&\\
\hspace{.4cm} \left.+\frac{(g-a)!(N+a+b+1-g)_{2g-a-b+1}}{M_{g-b}-1}h_{a+b-g-1}^{-a,-b,-2-N}(x)\right],&b\le g\le a+b-1,\vspace{.2cm}\\
 h_g^{-a,-b,-2-N}(x), & \mbox{otherwise}.
 \end{cases}
\end{equation}
Notice that only the polynomial $W_{i+b}^{a,b,N;\M}$ depends on the parameter $M_i$, $i=0,\cdots , a-1$.

The finite set $\{b,b+1,\cdots, a+b-1\}$ which appears in the determinants (\ref{defphi}) and (\ref{qusmei2}) remains the same (because of the same reasons
explained in Section 3 for the case $b\le a$).

With these changes, Theorem \ref{th1} works in the same way that for the case $b\le a$. More precisely:

\begin{theorem}\label{th1lq} Let $a,b,N$ be nonnegative integers with $1\le a\le b\le N$, and write $\M=\{M_0,\cdots, M_{a-1}\}$ for a finite set consisting of $a$ real parameters, $M_i\not =0,1$. Then the measure $\nu_{a,b,N}^{\M}$ (\ref{lanuf}) has a sequence of orthogonal polynomials if and only if
\begin{equation*}\label{cnysf}
\Phi_n^{a,b,N;\M}=\left|
  \begin{array}{@{}c@{}lccc@{}c@{}}
  &  &&\hspace{-.9cm}{}_{1\le j\le a} \\
    \dosfilas{W_{g}^{a,b,N;\M}(-n+j-1) }{g\in \{b,b+1,\cdots, a+b-1\}}
  \end{array}\hspace{-.3cm}\right| \not =0,\quad n=0,\cdots, N+a+1.
\end{equation*}
In that case the sequence of polynomials $(q_n^{a,b,N;\M})_{n=0}^{N+a}$ defined by
\begin{equation}\label{qusmei2f}
q_n^{a,b,N;\M}(x)=
\left|
  \begin{array}{@{}c@{}lccc@{}c@{}}
  & \frac{(-1)^{j-1}}{(b+N-n+j)_{a+1-j}}R_{n-j+1}^{a,b,N}(x) &&\hspace{-.6cm}{}_{1\le j\le a+1} \\
    \dosfilas{W_{g}^{a,b,N;\M}(-n+j-2) }{g\in \{b,b+1,\cdots, a+b-1\}}
  \end{array}\hspace{-.3cm}\right| ,
\end{equation}
is orthogonal  with respect to the measure $\nu_{a,b,N}^{\M}$, with norm
\begin{equation}\label{normqf}
\langle q_n^{a,b,N;\M},q_n^{a,b,N;\M}\rangle_{\nu_{a,b,N}^{\M}}=\frac{(N+b)!^2\Phi_n^{a,b,N;\M}\Phi_{n+1}^{a,b,N;\M}}{(N+a-n)!(N+b-n)!(N+b-n+1)_a^2}.
\end{equation}
Moreover, the polynomials $q_n^{a,b,N;\M}(\lambda^{a,b}(x))$, $n\ge 0$, are also eigenfunctions of a higher order difference operator of the form (\ref{hodo}) with $-s=r=ab+1$.
\end{theorem}
All the results in Sections \ref{sec3}, \ref{sec4} and \ref{sec5} can be adapted in the same form as Theorem \ref{th1}, with the only additional change of the assumptions (\ref{lacsu}) and (\ref{lacsu2}) in Section \ref{sec4} that have to be changed to
\begin{align*}
\lambda^{a,b}(u)&\in \{\lambda^{a,b}(i):\mbox{$-a\le i\le -2$ or $0\le i\le N$}\},\quad u\in U,\\
U&\subset \{i:-a-b+1\le i\le -b-1\}\cup \{i:1\le i\le N\},
\end{align*}
respectively.

\bigskip

\noindent
\textit{Mathematics Subject Classification: 42C05, 33C45, 33E30}

\noindent
\textit{Key words and phrases}: Orthogonal polynomials. Krall discrete polynomials.
Dual Hahn polynomials.

     \end{document}